\documentclass[ECP]{ejpecp} 

\usepackage{amsmath, amssymb, amsfonts, amsthm, bm, enumerate, mathtools, dsfont}

\SHORTTITLE{Nonparametric testing via partial sorting} 

\TITLE{Nonparametric testing via partial sorting}

\AUTHORS{%
K.~Bisewski\footnote{University of Lausanne, Lausanne, Switzerland.}
\and
H.M.~Jansen\footnote{University College Roosevelt, Middelburg, the Netherlands.}
\and
Y.~Nazarathy\footnote{The University of Queensland, St.\ Lucia, Queensland, Australia.}
}


\AMSSUBJ{62G10; 62G30} 
\AMSSUBJSECONDARY{62G20; 60F99} 

\SUBMITTED{X} 
\ACCEPTED{Y} 

\VOLUME{0}
\YEAR{0}
\PAPERNUM{0}
\DOI{10.1214/YY-TN}

\ABSTRACT{
\noindent In this paper we introduce the idea of partially sorting data to design nonparametric tests. This approach gives rise to tests that are sensitive to both the order and the underlying distribution of the data. We focus in particular on a test that uses the bubble sort algorithm to partially sort the data. We show that a function of the data, referred to as the empirical bubble sort curve, converges uniformly to a limiting curve. We define a goodness-of-fit test based on the distance between the empirical curve and its limit. The asymptotic distribution of the test statistic is a generalization of the Kolmogorov distribution. We apply the test in several examples and observe that it outperforms classical nonparametric tests for appropriately chosen sorting levels.
}

\usepackage{amsmath, amssymb, amsfonts, amsthm, bm, mathtools, dsfont}
\usepackage{enumerate}
\usepackage{xcolor}

\let\originalleft\left
\let\originalright\right
\renewcommand{\left}{\mathopen{}\mathclose\bgroup\originalleft}
\renewcommand{\right}{\aftergroup\egroup\originalright}

\newcommand{\abs}[1]{\left\vert #1 \right\vert}

\newcommand{\ind}[1]{\mathds{1}{\left\lbrace #1 \right\rbrace}}

\newcommand{\R}{\mathbb{R}}

\newcommand{\Z}{\mathbb{Z}}
\newcommand{\dd}{\mathrm{d}}
\newcommand{\opnot}[1]{\mathsf{#1}}
\newcommand{\p}{\mathbb{P}}
\newcommand{\e}{\mathbb{E}}

\DeclareMathOperator*{\cov}{\mathbb{C}ov}
\DeclareMathOperator*{\var}{\mathbb{V}ar}

\renewcommand{\tilde}{\widetilde}
\renewcommand{\hat}{\widehat}

\begin{document}
	
\section{Introduction}
Determining whether a data set constitutes a random sample from a given distribution is one of the central problems in nonparametric statistics. A common procedure is the standard Kolmogorov--Smirnov (KS) test for real-valued observations $X_1,\ldots,X_n$. 
The KS test considers the hypotheses
\begin{equation}
\label{eq:basic-hypothesis}
H_0 \colon X_i \underset{\text{iid}}{\sim} F_0,
\qquad
\text{vs.}
\qquad
H_1 \colon \text{otherwise}.
\end{equation}
Here $F_0$ is some distribution specified as part of $H_0$. The procedure computes the KS test statistic
\[
\hat{D}_n = \sqrt{n} \sup_{x \in {\mathbb R}} \big|\hat{F}_n(x) - F_0(x)\big|
\qquad
\text{with}
\qquad
\hat{F}_n(x) = \frac{1}{n}\sum_{i=1}^n \ind{X_i \le x},
\]
where $F_0$ is the cdf and $\hat{F}_n$ is the empirical cumulative distribution function (ecdf). 

The KS test does not prescribe or require any sorting of the observations. Even if $X_1,\ldots,X_n$ are fully sorted or partially sorted, this does not affect the test statistic $\hat{D}_n$. In other words, the KS test is insensitive to the order of the data. In this paper we generalize KS by introducing and studying a modified procedure that is based on partially sorting $X_1,\ldots,X_n$. The advantage of using partial sorting is that it makes the modified procedure sensitive to the order as well as the underlying distribution of the data. We study the properties of the modified procedure and provide examples in which KS is outperformed by our generalization. Before describing the main idea in more detail, let us review some aspects of KS.

For any continuous $F_0$ under $H_0$, as $n$ grows, the sequence $\hat{D}_n$ converges in distribution to a random variable $D$ that follows the Kolmogorov distribution with cdf
\begin{equation}
    \label{eq:kolmogorov-cdf}
{\mathbb P}(D \le x) = 1-2 \sum_{k=1}^{\infty}(-1)^{k-1} e^{-2 k^{2} x^{2}}.
\end{equation}
The quantiles of the Kolmogorov distribution are easily computed numerically, so the $p$-value or the test decision for a pre-specified confidence level $\alpha$ can be easily determined. The distribution of $D$ is derived based on a Brownian bridge approximation for the process $\sqrt{n} \big(\hat{F}_n(x) - F_0(x)\big)$. See for example \cite[Sec.\ 19.3]{vandervaart1998} for details or the initial publications \cite{kolmogorov1933sulla, smirnoff1939ecarts, smirnov1939estimation}. An overview of additional recent work related to KS tests is provided below.

\begin{figure}
\centering
\includegraphics[width=\linewidth]{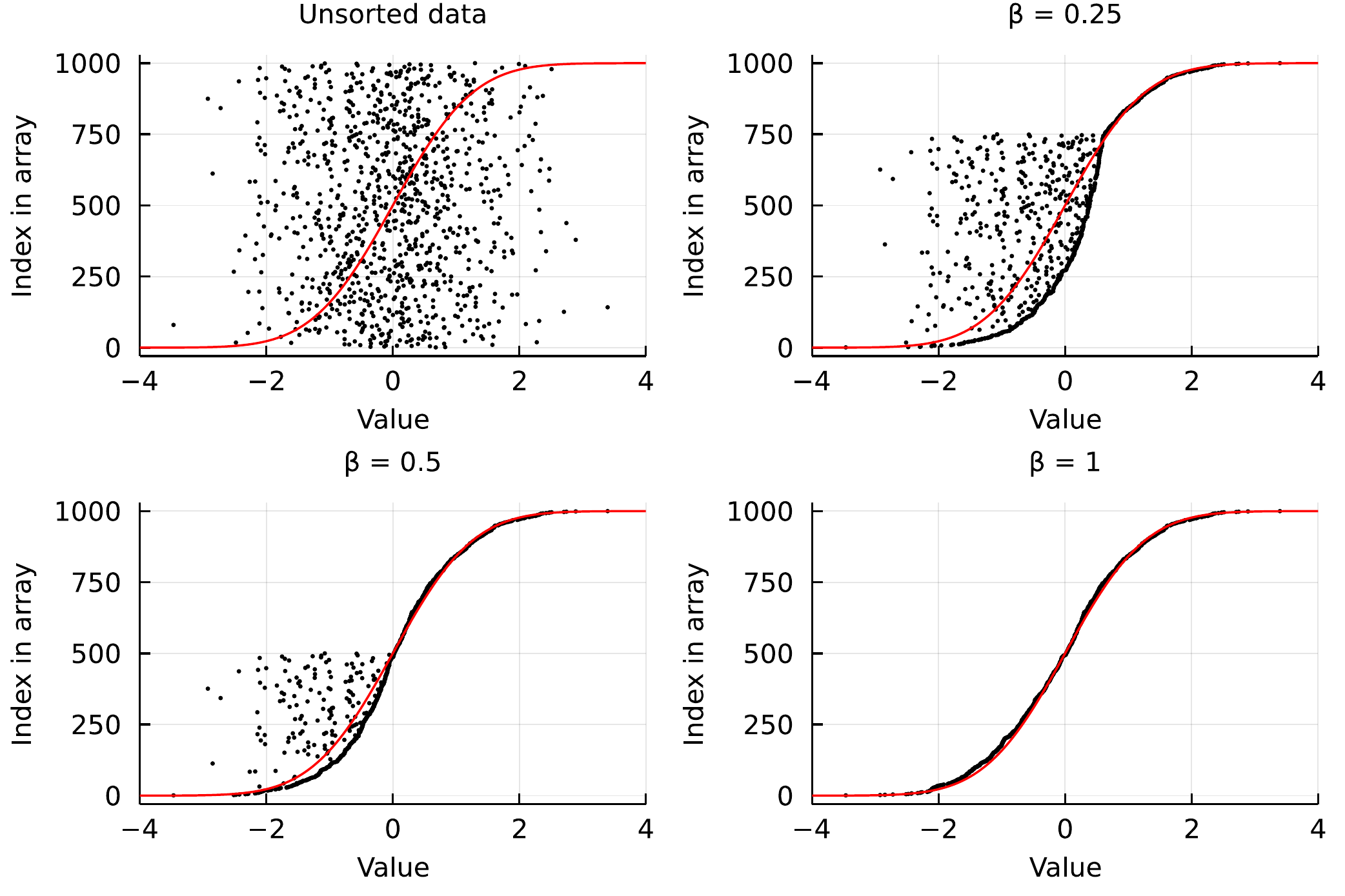}
\caption{iid\ normal data ($n=1,000$): Increasing levels of partial sorting with no sorted data (top left) to fully sorted data (bottom right). The scaled normal cdf is plotted in red.
}
\label{fig:basic_example}
\end{figure}

Importantly, KS tests are not only used for testing the distribution $F_0$ but also for the independence stated under $H_0$ or lack thereof. As an extreme example where this independence certainly does not hold, assume that $X_1=X_2=\ldots=X_n$, with $X_1$ distributed according to $F_0$. In this case
\[
\hat{D}_n = \sqrt{n} \max \big\{F_0(X_1),1-F_0(X_1)\big\},
\]
which is distributed uniformly in the range $[\sqrt{n}/2,\sqrt{n}]$. Now with a specified type~I error of $\alpha$, the null hypothesis $H_0$ is rejected almost with certainty as $n$ grows. For example, with $\alpha = 0.01$, the $0.99$th quantile of the Kolmogorov distribution is approximately $1.68$. Then $H_0$ is already rejected with certainty for $n = 12$, since $\sqrt{12}/2 \approx 1.73$. This extreme example indicates that in general KS statistics may be used to test for independence. Similarly, even if the lack of independence is not as extreme as having all observations identical, a KS test may often reject $H_0$ if the observations exhibit a form of dependence. It is thus common to use KS tests for testing lack of independence. 

Another important statistical test, focusing specifically on independence, is the Wald--Wolfowitz (WW) test \cite{wald1940test}. It is based on the runs statistic, which counts the number of times a data point above the median is followed by a data point below the median or vice versa. While both KS and WW may be used to test for lack of independence, these two tests focus on two different aspects of the data. The KS test takes the distribution of the data into account but is oblivious to its order. This means that it is sensitive to changes in the values of the data but not to changes in the order of the data. In contrast, the WW test relies on a statistic that takes the order of the data into account but is oblivious to the distribution. This means that it is sensitive to changes in the order of the data but not to changes in the values of the data, especially if data points remain on the same side of the median. 

Ideally, we would like to use a statistic that is sensitive to both aspects. For this we present a new nonparametric paradigm based on {\em partial sorting}. We believe that the ideas that we present here may lead to a new suite of tests that may detect in certain cases a lack of independence better than current methods. We develop and analyze one such specific test in this paper. We also present examples of cases where partial sorting based tests are valuable.

The overarching idea of partial sorting based tests is to apply a sorting algorithm to the data without running the algorithm to completeness. This leads to a partially sorted sample, say  $X_1^\beta,\ldots,X_n^\beta$, where $\beta \in (0,1]$ indicates the level of partial sorting and is referred to as the {\em sorting level}. In the extreme case of $\beta = 1$ the sample is fully sorted, while for lower values of $\beta$ the sample is less sorted. A test statistic is then computed based on the partially sorted sample $X_1^\beta,\ldots,X_n^\beta$. For the test statistic that we present here, there is a known asymptotic distribution under $H_0$ that is not influenced by $F_0$. Our results then yield a statistical test similar in nature to KS, but based on the partially sorted sample and often more sensitive to a lack of independence.

\begin{figure}
\centering
\includegraphics[width=\linewidth]{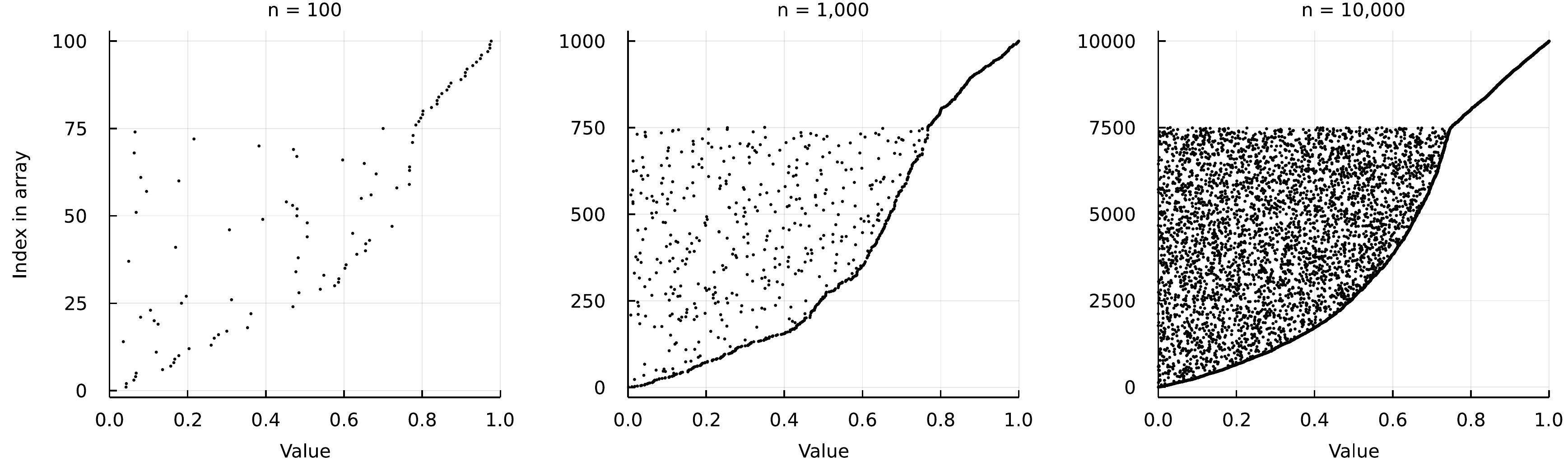}
\caption{{iid\ uniform data ($\beta = 0.25$): The empirical bubble sort curve is the right frontier of the points. It converges to a deterministic curve as $n \to \infty$.
}}
\label{fig:convergence_uniform}
\end{figure}

With $\beta < 1$ the partial sorting approach halts the sorting algorithm in mid flight. Thus the specific sorting algorithm at hand plays a key role in the analysis: different sorting algorithms generate different probability laws for $X_1^\beta,\ldots,X_n^\beta$. In this paper our sole focus is on partial sorting with the classic and simple {\em bubble sort} algorithm, which is a strong contender for being the most intuitive sorting algorithm. See for example \cite[Section~5.2.2]{knuth1997art} or \cite{astrachan2003bubble} for an historical overview. As its name may suggest, bubble sort works by `bubbling up' observations based on pairwise comparisons of adjacent values. In each case where adjacent observations are not in order they are swapped. Each such pass on the data is called an iteration and the algorithm applies $n$ iterations to yield a sorted sample. With bubble sort, a natural interpretation of $\beta$ is that the number of iterations performed is the nearest integer to~$\beta n$.

Our bubble sort based procedure generalizes the KS test, since it exactly agrees with KS if $\beta = 1$. Importantly, we show examples where using $\beta < 1$ outperforms both KS and WW. Our purpose with this paper is to show viability of this method together with a complete derivation of the properties of associated stochastic processes needed to obtain the asymptotic distribution of the test statistic. In doing so, we derive a generalization of the Kolmogorov distribution which we call the {\em generalized Kolmogorov distribution} with random variable denoted $D^\beta$. Tabulation of the cdf and/or quantiles of $D^\beta$ is simply a matter of computation similar to the Kolmogorov distribution, and like the Kolmogorov distribution our generalized Kolmogorov distribution is not influenced by $F_0$. Further, the underlying stochastic process that yields the derivation of $D^\beta$ is based on a generalization of the Brownian bridge process used for KS. Our analysis includes a Glivenko--Cantelli like (strong law of large numbers) theorem for this process, which is followed by a second order analysis establishing convergence of probability measures.

To get a feel for potential applicability of partial sorting with bubble sort consider Figure~\ref{fig:basic_example} which is based on an iid sequence of $n=1,000$ observations from a standard normal distribution. The figure presents the original data, followed by a display of the partially sorted data for $\beta = 0.25$, $\beta = 0.5$, and finally $\beta = 1$ (fully sorted). The nature of bubble sort is that with sorting level $\beta$ the highest valued $\beta n$ (rounded to an integer) observations are perfectly sorted in the respective top positions of the array but the remaining $(1-\beta)n$ observations are only partially sorted. Nevertheless as is evident from the figure a pattern emerges. Specifically observe the right frontier in the case of $\beta = 0.25$ and $\beta = 0.5$ and note that it follows a regular shape different from the cdf of the observations. We use this regular pattern to define the {\em bubble sort curve} in the next section and make use of it throughout the paper. In fact, we are able to characterize the bubble sort curve in terms of $F_0$ and $\beta$.

Our statistical method compares the bubble sort curve expected under $H_0$ with an empirical bubble sort curve obtained from the right frontier of the data. This operates in a manner that resembles the KS test. Remarkably, like in the case of the KS test we are able to characterize the deviations between the empirical bubble sort curve and the theoretical curve under $H_0$. Specifically, as $n \to \infty$, the scaled deviations between the curves converge to a well defined Gaussian process. We use this process to develop a test statistic with a distribution under $H_0$ which is agnostic to $F_0$ and thus directly generalizes the Kolmogorov distribution associated with the KS test.

To get a feel for the large sample asymptotics consider Figure~\ref{fig:convergence_uniform} which is based on standard iid uniform observations, each partially sorted with $\beta = 0.25$. The figure hints that as $n \to \infty$ the empirical bubble sort curve (the right frontier of the data) converges. In Theorem~\ref{thm:bubble_sort_lln} we establish the associated first order convergence and further in Theorem~ \ref{thm:bubble_sort_diffusion} we study the fluctuations around this frontier and characterize a limiting stochastic process of these fluctuations. Finally in Theorem~\ref{thm:gKS_convergence} we compute the distribution of a statistic arising from the fluctuations, similar to the KS statistic.

Similarly to the analysis that surrounds the KS test, our analysis is based both on a law of large numbers type theorem and a second order theorem dealing with convergence of probability measures. The former is an analogue of the celebrated Glivenko--Cantelli theorem, see \cite[Ch.~19]{vandervaart1998} and \cite{glivenko1933sulla}. The latter is analogous to Donsker's theorem for cdfs; see  \cite{donsker1952justification}. Hence our work is an extension of the basic development of the asymptotics of the KS test as in \cite{donsker1952justification}, \cite{doob1949heuristic}, and \cite{feller1948limittheorems}. Specifically we develop process level convergence that is similar in spirit to the process level convergence which appears in the analysis of the KS test.

A good historical account of Kolmogorov's original paper \cite{kolmogorov1933sulla} can be found in \cite{stephens1992appreciation}. Since that initial work, there have been dozens of extensions and generalizations, different in nature to our work here. We now mention a few notable publications in this space. In \cite{mason1983modified} the authors consider R\'enyi-type statistics (see \cite{renyi1973group}) and a adapt them to create modified KS tests that are sensitive to tail alternatives. In \cite{moscovich2016exact} the authors continue investigation of fitting tails of the distribution and study exact properties of the so-called $M_n$ statistics of Berk and Jones \cite{berk1979goodness}. Related is the recent work \cite{finner2018two} dealing with two-sample KS type tests, using local levels. In \cite{jager2007goodness} the authors consider goodness of fit tests via $\phi$-divergences. Beyond these works, many other papers have explored variations and adaptations of KS tests; see the notable papers \cite{cmiel2020intermediate, friedman1979multivariate, gontscharuk2017asymptotics, gontscharuk2015intermediates, gontscharuk2016goodness, janssen2000global, kim2015power}. Nevertheless, to the best of our knowledge, the use of partial sorting to enhance goodness of fit procedures, as we present here, is new.

The remainder of this paper is structured as follows. In Section~\ref{sec:gen-procedure} we define the bubble sort algorithm and the goodness of fit test that uses the partially sorted sample. This also includes the definition of the generalized Kolmogorov distribution. In Section~\ref{sec:mainResults} the main theoretical results of this paper our presented. Then sections \ref{sec:lln}, \ref{sec:weak-convergence}, and \ref{sec:analyzing} contain the derivations and proofs of the results. In Section~\ref{sec:lln} we establish the law of large numbers type convergence results. In Section~\ref{sec:weak-convergence} we establish the weak convergence results. In Section~\ref{sec:analyzing} these results are assembled to obtain the distribution of $D^\beta$. In Section~\ref{sec:numerical} we present numerical examples of our proposed procedure and we conclude in Section~\ref{sec:discussion}.

\section{The Test Procedure}
\label{sec:gen-procedure}

\begin{figure}
\centering
\includegraphics[width=14cm]{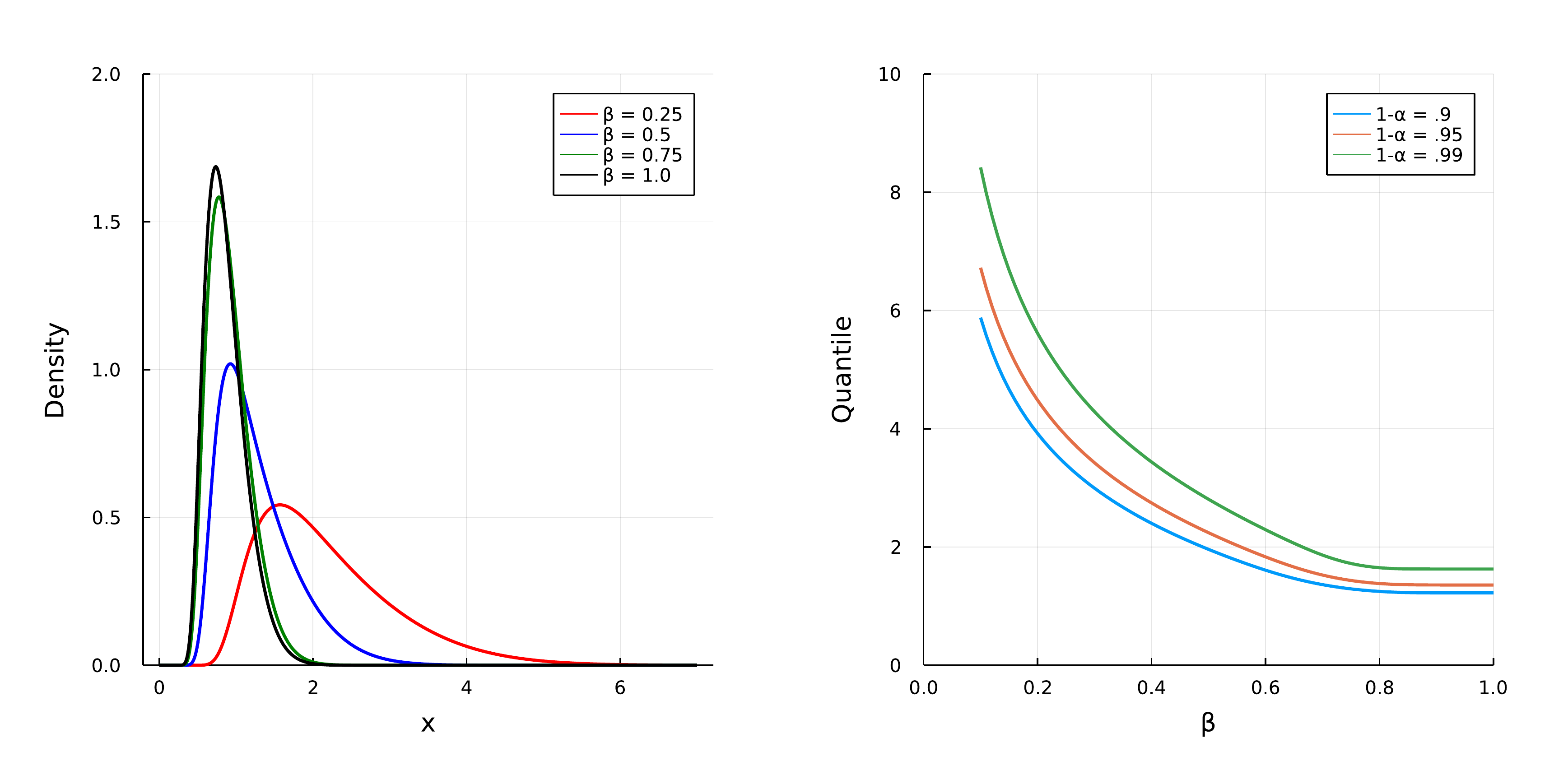}
 \caption{\label{fig:densities} Left: Densities of the generalized Kolmogorov distribution for several sorting levels (the case $\beta = 1$ is the Kolmogorov distribution). Right: Criticial value quantiles as a function of sorting levels.}
\end{figure}

{The bubble sort test considers the null hypothesis that the data consist of independent observations from a given continuous distribution $F_{0}$.} 
The procedure starts by partially sorting the data and computing the ecdf of the running maximum of the partially sorted data. We call this ecdf the empirical bubble sort curve. Then, similarly to the KS test, insight into the stochastic behavior of the empirical bubble sort curve under $H_0$ allows us to compare it to the bubble sort curve associated with $F_0$ and $\beta$. The distance between these two curves yields the bubble sort statistic. Finally, we compare the value of this statistic to the quantiles of its theoretical distribution under $H_0$, which is a generalized Kolmogorov distribution that only depends on $\beta$ and not on the underlying distribution $F_{0}$. See Figure~\ref{fig:densities} where we present densities for a few selected $\beta$ values as well as common quantiles as a function of $\beta$. Note that at $\beta = 1$ the distribution is the Kolmogorov distribution \eqref{eq:kolmogorov-cdf} and as is evident from the quantile plots, for $\beta$ in the approximate range of $0.7$--$1.0$ the generalized Kolmogorov distributions are very close to the Kolmogorov distribution. This is due to the fact that with bubble sort and such high $\beta$ values while the sample isn't guaranteed to be fully sorted it almost always is.

Steps~1--5 below describe the {\em bubble sort test procedure} for a data sample $x_1,\ldots,x_n$. Key objects in this description include the partially sorted sample \eqref{eq:partial-bs}, the empirical bubble sort curve \eqref{eq:emp-bubble-sort-curve}, the limiting bubble sort curve \eqref{eq:bubble-sort-curve1}, the bubble sort statistic \eqref{eq:bubble-sort-stat}, and the generalized Kolmogorov distribution \eqref{eq:gen-kolg-dist}. This distribution is defined via probabilities of the form \eqref{eq:gen-bb-prob} associated with Brownian bridges.
\newpage
\begin{itemize}
    \item[\textbf{1}]
    \textbf{Partial sorting:}
    Take $k$ as the nearest integer to $\beta n$ and apply $k$ bubble sort iterations to the data. The result is the partially sorted sample $x_1^\beta,\dotsc,x_n^\beta$. A bubble sort iteration on a data array $v$ with elements $v[1],\ldots,v[n]$ is defined via the operator $\opnot{T}$, which is the composition $\opnot{S}_{n-1} \circ \dotsb \circ \opnot{S}_{1}$ with
    \begin{align*}
        \opnot{S}_{i} (v) =
            \begin{cases}
                    \big(v[1] , \dotsc , v[i-1], v[i+1], v[i], v[i+2], \dotsc , v[n]\big) & \text{ if } v[i] > v[i+1], \\
                    \big(v[1] , \dotsc , v[n]\big) & \text { else. }
            \end{cases}
    \end{align*}
    The partially sorted sample $x_1^\beta,\dotsc,x_n^\beta$ is then given by
    \begin{align}
        \label{eq:partial-bs}
    (x_1^\beta,\ldots,x_n^\beta) = \opnot{T}^k(x_1,\ldots,x_n).
    \end{align}
    \item[\textbf{2}]
    \textbf{Empirical bubble sort curve:}
    Obtain the {\em empirical bubble sort curve} $\hat{B}_n^\beta$ via
    \begin{align}
        \label{eq:emp-bubble-sort-curve}
        \hat{B}_n^\beta (x) = \frac{1}{n} \sum_{i = 1}^{n} \ind{\max \lbrace x_1^\beta, \dotsc, x_i^\beta \rbrace \leq x}.
    \end{align}
    This is the ecdf of the running maximum of the partially sorted sample.
    \item[\textbf{3}]
    \textbf{Bubble sort curve:}
    Determine the {\em bubble sort curve} $B_0^\beta$ associated with the distribution $F_0$ and sorting level $\beta$ via
    \begin{align}
        \label{eq:bubble-sort-curve1}
        B_0^\beta(x) = 
        \min\Big\{\frac{\beta}{1-F_0(x)}, 1\Big\} - \min\big\{\beta,1-F_0(x)\big\}.
    \end{align}
    \item[\textbf{4}]
    \textbf{Bubble sort statistic:}
    Compute the {\em bubble sort statistic} $\hat{D}_n^\beta$ via
    \begin{align}
        \label{eq:bubble-sort-stat}
        \hat{D}_n^\beta = \sup_{x \in \mathbb{R}} \sqrt{n} \big|\hat{B}_n^\beta (x) - B_0^\beta(x)\big|.
    \end{align}
    \item[\textbf{5}]
    \textbf{Test conclusion:}
    Compare the value of the bubble sort statistic $\hat{D}_n^\beta$ to the quantiles of its asymptotic distribution $D^\beta$ under $H_0$. This is a generalized Kolmogorov distribution with cdf
    \begin{align}
    \label{eq:gen-kolg-dist}
        \p(D^\beta \leq x) = 2\int_0^{\sqrt{\tfrac{\beta}{1-\beta}} x} \Psi \Big( x; \tfrac{1-\beta}{\beta}, \sqrt{\tfrac{1-\beta}{\beta}} z \Big) \Psi \Big( x;\beta, \sqrt{\beta(1-\beta)} z \Big) \phi(z) \, \dd z,
    \end{align}
    where
    \begin{align}
        \label{eq:gen-bb-prob}
    \Psi (x;T,a) = 
        \begin{cases}
            \sum_{k\in\Z} (-1)^k \exp \Big(-\frac{2kx(kx-a)}{T} \Big) & \text{ if } x > |a|,\\
            0 & \text{ if } x \leq |a|,
        \end{cases}
    \end{align}
    and $\phi$ is the standard normal density function.
    The resulting asymptotic $p$-value is $\p(D^\beta > \hat{D}_n^\beta)$ provided $F_{0}$ is continuous.
\end{itemize}

In the boundary case $\beta = 1$, this procedure reduces to the KS test. Although the expression in \eqref{eq:gen-kolg-dist} is not defined for $\beta = 1$, it is readily verified that, as $\beta \nearrow 1$, the test statistic $\hat{D}_n^\beta$ converges weakly to the Kolmogorov distribution with cdf~\eqref{eq:kolmogorov-cdf}. Tabulation of the cdf and quantiles for generalized Kolmogorov distributions of the form~\eqref{eq:gen-kolg-dist} is computationally straightforward. See also the accompanying GitHub repository for this paper~\cite{bs-github}. 

Figure~\ref{fig:test-illustration} illustrates the procedure via $n=100$ iid\ standard uniform random variables with sorting level $\beta = 0.25$. The associated empirical bubble sort curve and the limiting bubble sort curve are plotted in the left display. The scaled difference between the curves is plotted in the right display and is used to compute the test statistic. In this example the value of the test statistic is $1.598$. Based on the generalized Kolomogorov distribution for $\beta = 0.25$ (see Figure~\ref{fig:densities}), the corresponding asymptotic $p$-value is $0.701$. Thus $H_0$ is not rejected. Further examples are in Section~\ref{sec:numerical}.

\begin{figure}
\centering
\includegraphics[width=14cm]{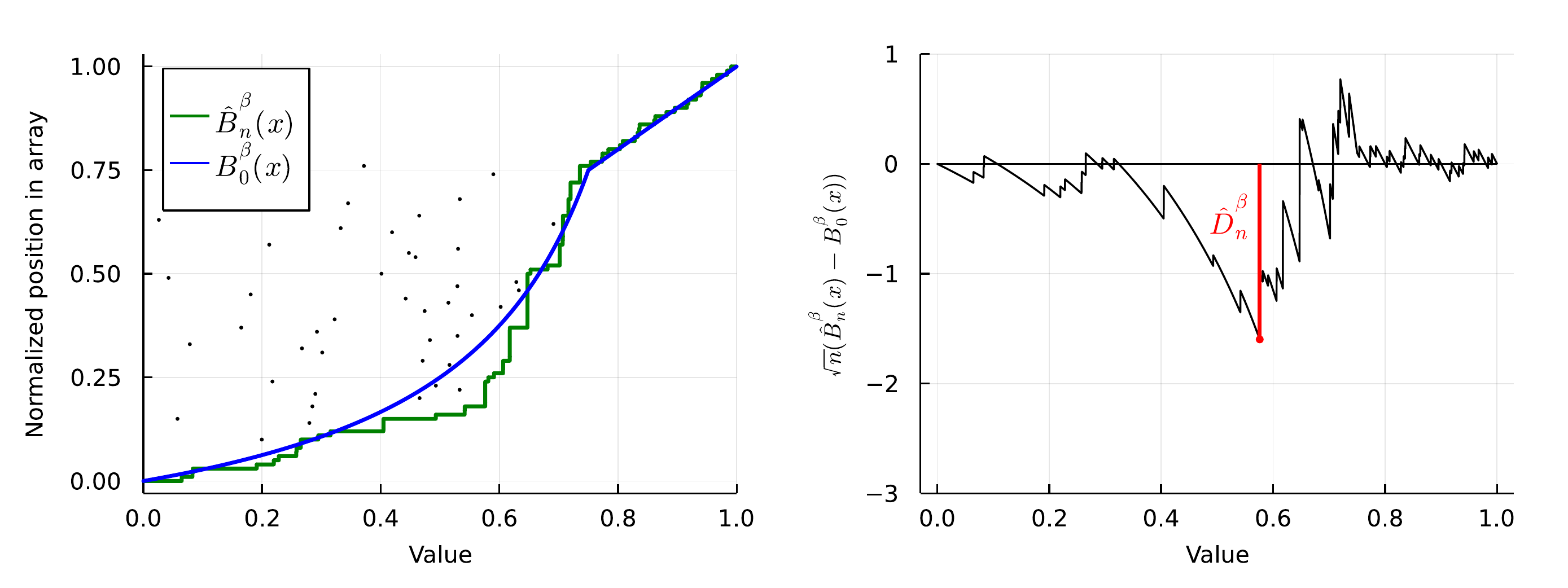}
\caption{\label{fig:test-illustration}An illustration of the test procedure with $\beta = 0.25$. Left: The empirical bubble sort curve (green) together with the bubble sort curve (blue). Right: The scaled difference between the curves with the maximal absolute value marking the test statistic $\hat{D}_n^\beta$.}
\end{figure}

\section{Main Results}
\label{sec:mainResults}

We now present the main results supporting the validity of the bubble sort test procedure. Our context in this section and the three sections that follow is probabilistic. We therefore avoid any specific references to the test procedure and instead focus on the probabilistic results valid under $H_0$. In the remainder we fix a sorting level $\beta$ in $(0,1]$ and a distribution $F_0$ for the iid\ random variables $X_1,\dotsc,X_n$. Some of the results below hinge on assuming $F_0$ is continuous, and we state this explicitly when needed.

The key object we investigate is the bubble sort curve \eqref{eq:bubble-sort-curve1}, which can also be more explicitly represented via
\begin{align}
\label{eq:bubble_sort_curve}
B_0^\beta(x) = 
\begin{cases}
\beta \frac{F_0(x)}{1 - F_0(x)} & \text{ if } x < x^{*}_\beta,\\
F_0(x) & \text{ else,}
\end{cases}
\qquad
\text{where}
\qquad
x^{*}_\beta = F_0^{-1} (1 - \beta). 
\end{align}

Observe that since for $x < x^{*}_\beta$ the bubble sort curve is the scaled odds of the probability $F(x)$, we can recover the cdf $F_0(\cdot)$ from $B_0(\cdot)$ using $F_0(x) = B_0^\beta(x) / (\beta + B_0^\beta(x))$ for such $x$. Thus, analogous to the cdf, the bubble sort curve is a nondecreasing, right-continuous function that fully characterizes its underlying distribution. 

The empirical bubble sort curve, $\hat{B}_n^\beta(\cdot)$ defined in \eqref{eq:emp-bubble-sort-curve}, plays the same role with respect to $B_0^\beta(\cdot)$ as the ecdf plays with respect to the cdf. Here it is constructed as a random process with index $x \in {\mathbb R}$ based on the random sample $X_1,\ldots,X_n$,
\[
\hat{B}_n^\beta(x) = \frac{1}{n} \sum_{i=1}^n \ind{\max \lbrace X_1^\beta, \dotsc, X_i^\beta \rbrace \leq x},
\quad
\text{where}
\quad
X_1^\beta,\ldots,X_n^\beta = \opnot{T}^k(X_1,\ldots,X_n),
\]
and with $k$ the nearest integer to $\beta n$ and $\opnot{T}$ defined in \eqref{eq:partial-bs}.

The following theorem is analogous to the Glivenko--Cantelli theorem. Note that we recover the Glivenko--Cantelli theorem by taking $\beta = 1$.  
Section~\ref{sec:lln} is devoted to the proof.
\begin{theorem}
\label{thm:bubble_sort_lln}
The uniform distance $\sup_{-\infty < x < \infty} \vert \hat{B}^\beta_{n} (x) - B^\beta_0(x) \vert \to 0$ almost surely as $n \to \infty$.
\end{theorem}

Now, analogous to the Donsker theorem underpinning the Kolmogorov--Smirnov theorem, we wish to characterize the fluctuations of $\hat{B}_{n}^\beta$ around its limiting curve $B^\beta_0$. We do this in the following theorem which coincides with the Donsker theorem for empirical cumulative distribution functions and Brownian bridges if $\beta = 1$, yet generalizes it with $\beta < 1$. Section~\ref{sec:weak-convergence} is devoted to the proof.

\begin{theorem}
\label{thm:bubble_sort_diffusion}
If $F_0$ is continuous then the process $\sqrt{n} (\hat B_{n}^\beta - B^\beta_0)$ converges weakly to $Y^\beta_0$ in the M1 topology where $Y^\beta_0$ is  right-continuous, zero mean, Gaussian process with covariance function
\begin{align}
\label{def:Y_covariance}
c(x,y) = \cov\big(Y_0^\beta(x),Y_0^\beta(y)\big) =
\begin{cases}
\beta \frac{F_0(x)}{\big(1-F_0(x)\big)^2} & \text{ if } x \leq y < x^{*}_\beta, \\[5pt]
\beta \frac{F_0(x)\big(1-F_0(y)\big)}{\big(1-F_0(x)\big)^2} & \text{ if } x < x^{*}_\beta \leq y, \\[5pt]
F_0(x)\big(1-F_0(y)\big) & \text{ if } x^{*}_\beta \leq x \leq y.
\end{cases}
\end{align}
\end{theorem}

Although the process $Y^\beta_0$ might seem unfamiliar at first, it is closely related to generalized Brownian bridges, which we discuss below. For this also consider Figure~\ref{fig:discontinuity_uniform}. The covariance function of $Y^\beta_0$ and the realizations in Figure~\ref{fig:discontinuity_uniform} reveal the presence of a jump at $x^{*}_\beta$. To the left of $x^{*}_\beta$, the process $Y^\beta_0$ behaves like a time-scaled Brownian motion. To the right of $x^{*}_\beta$, the process $Y^\beta_0$ behaves like a time-scaled Brownian bridge. The crucial property of $Y^\beta_0$ is that, conditional on the jump at $x^*_\beta$, on each side of $x^{*}_\beta$ it behaves like an independent generalized Brownian bridge under a time scaling.

We say that a stochastic process $\lbrace W(t) : t \in [0,T] \rbrace$ follows the law of a generalized Brownian bridge from $(0,0)$ to $(T,a)$ and write $W \sim {\rm BB}(T,a)$ if $W$ is a continuous Gaussian process with mean function $\mathbb{E} \big(W(t)\big) = \frac{ta}{T}$ and covariance function $\cov \big(W(s),W(t)\big) = \frac{s(T-t)}{T}$ for $0 \leq s \leq t \leq T$. Note that a standard Brownian motion on $[0,T]$ conditioned on reaching $a$ at time $T$ follows the law of ${\rm BB}(T,a)$. In particular, ${\rm BB}(1,0)$ is the law of a standard Brownian bridge on $[0, 1]$.
An important probability associated with general Brownian bridges is
\begin{equation}
\label{eq:psi-x-t-a-def}
\Psi (x;T,a) = \p \big(\sup_{t\in[0,T]} |W(t)| \leq x\big),
\end{equation}
which can be represented as the the double infinite series expression  \eqref{eq:gen-bb-prob}. See for example \cite[Eq.~(4.12)]{beghin1999maximum} for a derivation. Note that $\Psi(x; 1, 0)$ recovers the Kolmogorov distribution \eqref{eq:kolmogorov-cdf}.

\begin{figure}
\centering
\includegraphics[width=\linewidth]{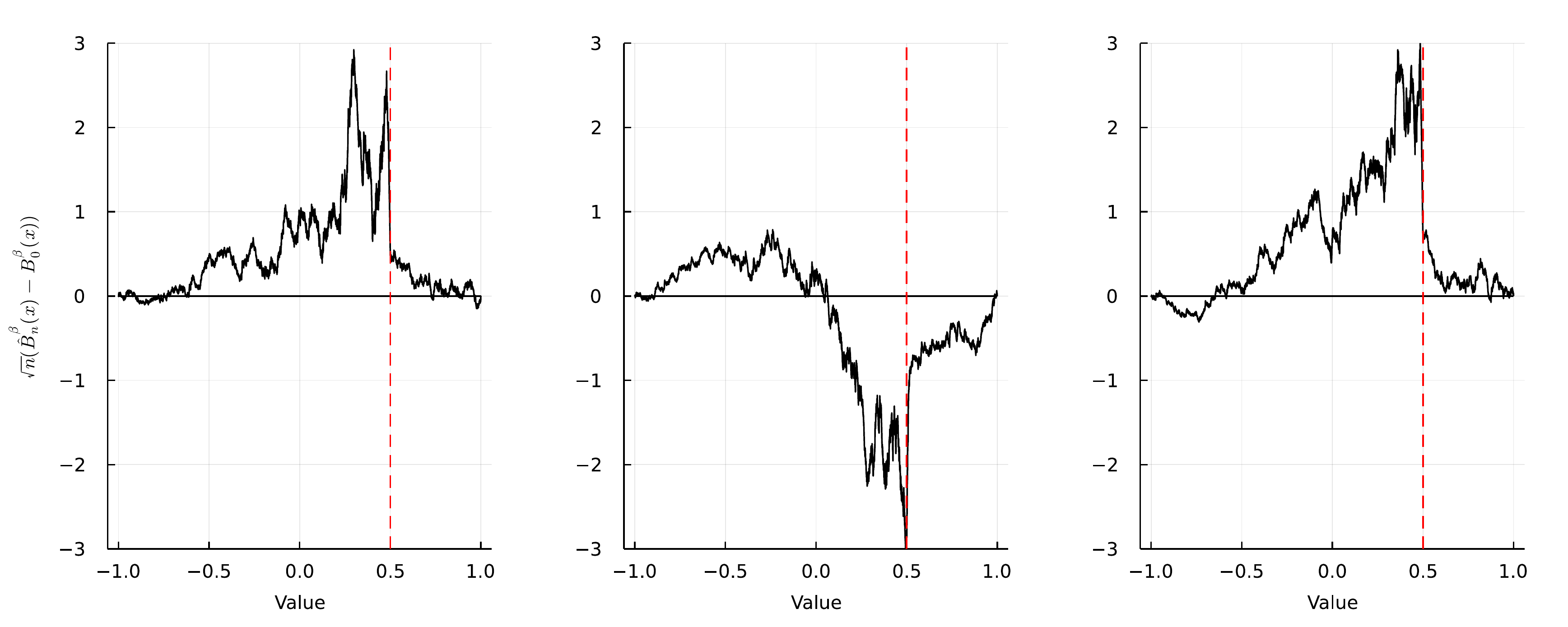}
\caption{\label{fig:discontinuity_uniform}Three random outcomes of $\sqrt{n}(\hat B_n^\beta(x)-B^\beta_0(x))$ with $\beta = 0.25$ and $n=10,000$, where all elements of the initial array are iid, uniformly distributed over $[-1,1]$. The vertical dashed line is at $x^*_\beta =0.5$. 
}
\end{figure}

With generalized Brownian bridges at hand we now construct a {\em generalized Kolmogorov distribution}. For this let $Z$ be a random variable with a standard normal distribution. Conditionally on $Z$, we let the processes $W_1$ and $ W_2$ be independent with laws
\begin{equation}
\label{eq:law-w1-w2}
W_1 \sim {\rm BB}\left(\tfrac{1-\beta}{\beta},\sqrt{\tfrac{1-\beta}{\beta}} Z\right) \quad \text{and} \quad W_2 \sim {\rm BB}\left(\beta,\sqrt{\beta(1-\beta)}Z\right).
\end{equation}

We define the distribution of
\begin{equation}
\label{eq:D-beta-prob-def}
D^\beta = \sup_{t \in [0, (1 - \beta)/\beta]} |W_1(t)| ~~\vee~~ \sup_{t \in [0, \beta]} |W_2(t)|
\end{equation}
as a generalized Kolmogorov distribution depending on the parameter $\beta$. Using this construction we have an expression for the distribution of $D^\beta$ in \eqref{eq:gen-kolg-dist}. The proof is in Section~\ref{sec:analyzing}.

\begin{lemma}
\label{lem:supremum_cdf}
The cdf of $D^\beta$ as defined in \eqref{eq:D-beta-prob-def} is given in \eqref{eq:gen-kolg-dist}.
\end{lemma}

Our final result deals with the asymptotic distribution of the bubble sort statistic \eqref{eq:bubble-sort-stat} and hinges on the process level convergence of Theorem~\ref{thm:bubble_sort_diffusion}. Section~\ref{sec:analyzing} is devoted to the proof.

\begin{theorem}
\label{thm:gKS_convergence}
If $F_0$ is continuous then the bubble sort statistic $\hat{D}^\beta_n$ of \eqref{eq:bubble-sort-stat} converges in distribution to 
$D^\beta$ of \eqref{eq:D-beta-prob-def}.
\end{theorem}

We emphasize that the generalized Kolmogorov distribution depends only on the sorting level $\beta$ and is insensitive to the underlying distribution $F_0$. Similarly, the distribution of $\hat{D}^\beta_n$ for any finite $n$ also does not depend on the underlying distribution of $X_i$. As an illustration of the convergence in distribution of Theorem~\ref{thm:gKS_convergence} consider Figure~\ref{fig:stat-conv} focusing on the bubble sort statistic with three sorting levels $\beta$. The black curves in all three plots are identical and are the cdf of $D^\beta$ computed via \eqref{eq:gen-kolg-dist}. In all three cases we consider sample sizes of $n=20$, $n=200$, and $n=2,000$. In each of these cases we simulate $30,000$ repetitions of the samples to obtain the bubble sort statistic \eqref{eq:bubble-sort-stat} and plot its estimated (finite $n$) distribution. It is visually apparent that as $n$ grows, the distributions converge to the generalized Kolmogorov distribution. As is further apparent, for lower sorting levels $\beta$ the convergence is slower.

\begin{figure}
\centering
\includegraphics[width = \linewidth]{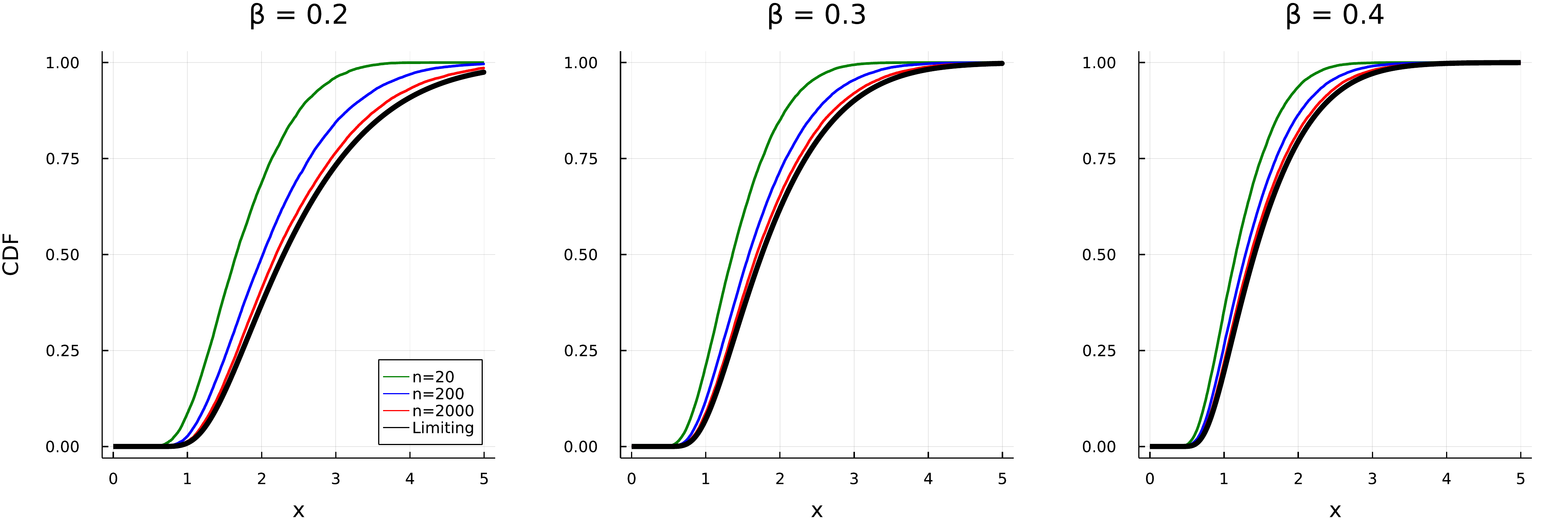}
\caption{\label{fig:stat-conv}Convergence in distribution of the bubble sort statistic $\hat{D}^\beta_n$ to the generalized Kolmogorov distribution $D^\beta$ for various sorting levels. Note that the distributions of $\hat{D}^\beta_n$ are not influenced by the distribution of $X_i$. 
}
\end{figure}

\section{The Law of Large Numbers}
\label{sec:lln}

This section is devoted to the proof of Theorem~\ref{thm:bubble_sort_lln}. More precisely, we consider the random array $V_{n}$ and study the frontier $\hat{B}_{n}^\beta$ associated with the partially sorted array $V_{n}^{\beta n}$. Our goal is to show that the frontier $\hat{B}_{n}^\beta$ converges uniformly to its limiting bubble sort curve $B_0^\beta$ defined in \eqref{eq:bubble_sort_curve}. In the remainder of this section we omit the superscript $\beta$ and the subscript $0$, and with this notation our goal is to show that $\hat{B}_n$ converges uniformly to $B$.

Before we prove this result, we explore how bubble sort operates on binary arrays, which contain only 0's and 1's. The next lemma explains how bubble sort changes the positions of the 1's. We demonstrate later that convergence properties of general random arrays are connected to convergence properties of related binary arrays.

\begin{lemma}
\label{lem:binary_positions}
Consider a binary array $U_{n}$. Assume that exactly $m$ of its entries equal $1$ and let $I_{n, 1} < \dotsc < I_{n, m}$ be the positions of the 1's in $U_{n}$. Denote by $I_{n, 1}^{k} < \dotsc < I_{n, m}^{k}$ the positions of the 1's in the partially sorted array $U_{n}^{k}$. Then
\begin{align*}
I_{n, i}^{k} &=
\begin{cases}
I_{n, i + k} - k & \text{ if } \quad k \leq m - i,\\
n - (m - i) & \text{ if } \quad k > m - i,
\end{cases}
\qquad
\text{for} \qquad i=1,\ldots,m.
\end{align*}
\end{lemma}
\begin{proof}
We start by focusing on the 1 at position $I_{i}$ with $i < m$ in the array $U_{n}$ during the first bubble sort iteration. Clearly, the 1 at $I_{i}$ does not get swapped with any values to the left of position $I_{i}$, because all values in $U_{n}$ are at most 1. We now distinguish two cases. If the value to the right of $I_{i}$ equals 1, then $I_{i + 1} = I_{i} + 1$. In this case, the 1 at position $I_{i}$ does not get swapped with any value, so $I_{i}^{1} = I_{i} = I_{i + 1} - 1$. If the value to the right of $I_{i}$ does not equal 1, then $I_{i + 1} > I_{i} + 1$. In this case, the 1 at position $I_{i}$ gets swapped with all 0's between position $I_{i}$ and $I_{i + 1}$, so it reaches the position on the left of $I_{i + 1}$. However, it does not swap with the 1 at position $I_{i + 1}$, so $I_{i}^{1} = I_{i + 1} - 1$ in this case too. Therefore, the statement of the lemma holds for $k = 1$. This may be extended to all $k \leq m - i$ by induction.

We now turn to the case $k > m - i$. The previous arguments also show that one bubble sort iteration moves every 1 to the position of the last 0 in the sequence of 0's immediately right to it. That means that $I_{n, m}^{k} = n$ for $k > 0$ and $I_{n, i}^{k} = n - (m - i)$ for $k > m - i$ by induction.
\end{proof}
Lemma~\ref{lem:binary_positions} enables us to relate the position $I_{1}$ of the leftmost 1 in $U_{n}$ to the position $I_{1}^{k}$ in $U_{n}^{k}$, which  we obtain after applying $k$ bubble sort iterations to $U_{n}$. This is important, because understanding the dynamics of the leftmost 1 in $U_{n}$ is crucial for understanding the behavior of its frontier.

As mentioned before, our goal is to connect the convergence of a general random array to convergence of a related binary random array. More concretely, given the random array $V_{n}$, we fix $x \in \mathbb{R}$ such that $0 < F(x) < 1$ and define the new binary array $U_{n}$ via
\begin{align*}
U_{n}[i] = 
\begin{cases}
0 & \text{ if } \quad V_{n}[i] \leq x,\\
1 & \text{ if } \quad V_{n}[i] > x.
\end{cases}
\end{align*}
The idea is that the position of the leftmost value in $V_{n}^{k}$ exceeding $x$ has the same position as the leftmost $1$ in $U_{n}^{k}$, regardless of the number $k$ of bubble sort iterations that we apply to both arrays. We formalize this in the next lemma and use it in combination with Lemma~\ref{lem:binary_positions} to establish the convergence of $\hat{B}_{n}$ to $B$.

\begin{lemma}
\label{lem:identical_positions}
For each $k = 0, 1, 2, \ldots,n$, it holds that
\begin{align*}
U_{n}^{k}[i] = 
\begin{cases}
0 & \text{ if } \quad V_{n}^{k}[i] \leq x,\\
1 & \text{ if } \quad V_{n}^{k}[i] > x,
\end{cases}
\end{align*}
so in particular
\begin{align*}
\inf \lbrace i = 1, \dotsc, n ~ \vert ~ V_{n}^{k}[i] > x \rbrace
&= \inf \lbrace i = 1, \dotsc, n ~ \vert ~ U_{n}^{k}[i] = 1 \rbrace.
\end{align*}
\end{lemma}
\begin{proof}
The statement of the lemma is trivially true if $V_{n}$ does not contain any value exceeding $x$ or if $k = 0$, so we assume that this is not the case. We only show that the statement holds for $k = 1$, since the case $k > 1$ follows by induction.

To establish the claim for $k = 1$, it suffices to show that applying an arbitrary swap operator $\opnot{S}_{j}$ to both $V_{n}$ and $U_{n}$ yields $\opnot{S}_{j} (V_{n})[i] > x$ if and only if $\opnot{S}_{j} (U_{n})[i] = 1$. Since by construction $V_{n}[i] > x$ if and only if $U_{n}[i] = 1$, this can only be possibly violated if $U_{n}[i] = 1$, $U_{n}[i+1] = 0$ and $j = i$. However, this means that $V_{n}[i] > x$ and $V_{n}[i+1] \leq x$, so $\opnot{S}_{i} (V_{n})[i] = V_{n}[i+1] \leq x$ and $\opnot{S}_{i} (V_{n})[i+1] = V_{n}[i] > x$ with $\opnot{S}_{i} (U_{n})[i] = U_{n}[i+1] = 0$ and $\opnot{S}_{i} (U_{n})[i+1] = U_{n}[i] = 1$ in this case. This proves the claim.
\end{proof}

We can use the previous two results to establish the convergence of $\hat{B}_{n} (x)$ to $B(x)$ as follows. We recall that $\beta$ is the sorting level and that the entries of $V_{n}$ consist of the first $n$ elements of the sequence of iid random variables $X_{1}, X_{2}, \dotsc$. Thus the entries of $U_{n}$ consist of the first $n$ elements of the sequence of iid random variables $\ind{X_{1} > x}, \ind{X_{2} > x}, \dotsc$. The latter each follows a Bernoulli distribution with success probability $p = 1 - F(x) \in (0, 1)$.

Since $\ind{X_{1} > x}, \ind{X_{2} > x}, \dotsc$ are iid, the index of the first 1 in that sequence has a geometric distribution with success probability $p$. The same is true for the distance between two consecutive 1's in that sequence. Thus the positions of the 1's in the sequence $\ind{X_{1} > x}, \ind{X_{2} > x}, \dotsc$ give rise to a sequence of iid random variables $G_{1}, G_{2}, \dotsc$ that have a geometric distribution with parameter $p$.

The total number of 1's in the binary array $U_{n}$ is given by
\begin{align*}
H_{n} = \sum_{i = 1}^{n} \ind{X_{i} > x}
\end{align*}
and has a binomial distribution with parameter $p$. In line with the notation used in Lemma~\ref{lem:binary_positions}, we also define
\begin{align*}
I_{k} &= \sum_{i = 1}^{k} G_{i}
\end{align*}
for $k = 1, 2, \dotsc$. The random variable $I_{k}$ simply denotes the position of the $k$-th 1 in the sequence $\ind{X_{1} > x}, \ind{X_{2} > x}, \dotsc$. This should be distinguished from the random variable $I_{n, k}$, which denotes the position of the $k$-th 1 in $U_{n}$.

Applying the results of Lemma~\ref{lem:binary_positions} and Lemma~\ref{lem:identical_positions} combined with our definition of $I_{k}$, we see that the position $I_{n, 1}^{\beta n}$ of the first 1 in the partially sorted array $U_{n}^{\beta n}$ satisfies
\begin{align*}
I_{n, 1}^{\beta n} = 
\begin{cases}
I_{1 + \beta n} - \beta n & \text{ if }\quad \beta n \leq H_{n} - 1,\\
n - H_{n} + 1 & \text{ if }\quad \beta n > H_{n} - 1,\\
\end{cases}
\end{align*}
so
\begin{align*}
I_{n, 1}^{\beta n} = \min \lbrace I_{1 + \beta n}, n \rbrace - \min \lbrace \beta n, H_{n} - 1 \rbrace.
\end{align*}
Indeed, if $I_{1 + \beta n} \leq n$, then there are at least $1 + \beta n$ entries of $U_{n}$ that contain a 1. Then the number of bubble sort iterations $\beta n$ that we apply to $U_{n}$ is smaller than the total number of 1's in $U_{n}$, so Lemma~\ref{lem:binary_positions} implies that $I_{n, 1}^{\beta n} = I_{1 + \beta n} - \beta n$ in this case. If $I_{1 + \beta n} > n$, then $U_{n}$ has at most $\beta n$ entries that contain a 1, so $H_{n} \leq \beta n$ and $I_{n, 1}^{\beta n} = n - (H_{n} - 1)$ by Lemma~\ref{lem:binary_positions} in this case.

We now combine these results to establish the convergence of $\hat{B}_{n} (x)$ to $B(x)$. As defined in \eqref{eq:emp-bubble-sort-curve}, the frontier $\hat{B}_{n}$ is given as the empirical cdf of the running maximum 
of $V_{n}^{\beta n}$. Thus $\hat{B}_{n}(x)$ is equal to the relative number of entries in the running maximum 
that do not exceed $x$. Clearly, the first value that does exceed $x$ in the running maximum 
has the same position as the first value exceeding $x$ in $V_{n}^{\beta n}$. Additionally, due to Lemma~\ref{lem:identical_positions}, the first value exceeding $x$ in $V_{n}^{\beta n}$ has the same position as the first 1 in $U_{n}^{\beta n}$. Then $I_{n, 1}^{\beta n}$ is also the position of first value exceeding $x$ in the running maximum.  
Consequently, $\hat{B}_{n} (x)$ satisfies 
\begin{align}
\label{eq:Bnx}
\hat{B}_{n} (x) = \frac{1}{n} (I_{1, n}^{\beta n} - 1) = \min \Big\lbrace \frac{1}{n} I_{1 + \beta n}, 1 \Big\rbrace - \min \Big\lbrace \frac{1 + \beta n}{n}, \frac{1}{n} H_{n} \Big\rbrace.
\end{align}
Since $\frac{1}{n} I_{1 + \beta n} \to \beta / p$ and $\frac{1}{n} H_{n} \to p$ almost surely, the continuous-mapping theorem implies that
\begin{align*}
\hat{B}_{n} (x) \to \min \lbrace \beta / p, 1 \rbrace - \min \lbrace \beta, p \rbrace
\end{align*}
almost surely. Taking into account that $\beta < p$ if and only if $\beta / p < 1$ and that $p = 1 - F(x)$, we conclude that $\hat{B}_{n}(x)$ converges to $B(x)$ almost surely. This can be extended to almost sure uniform convergence of $\hat{B}_{n}$ to $B$ by the standard Glivenko--Cantelli arguments (cf.\ \cite[Th.~19.1]{vandervaart1998}).

\section{Weak Convergence}
\label{sec:weak-convergence}

This section is devoted to the proof of Theorem~\ref{thm:bubble_sort_diffusion}. The proof has a classical structure: we first establish weak convergence of the finite-dimensional distributions (fdds) of the scaled and centered frontier $\tilde{B}_{n} = \sqrt{n} (\hat{B}_{n} - B)$ on a dense subset of the real line. It turns out that the fdds of $\tilde{B}_{n}$ converge to the fdds of a process that behaves like a time-scaled Brownian motion on $(-\infty, x^{*}_{\beta})$ and like a time-scaled Brownian bridge on $(x^{*}_{\beta}, \infty)$, with a jump occuring at $x^{*}_{\beta}$. We then proceed to show that $\tilde{B}_{n}$ is tight in the M1-topology. These two steps combined imply the weak convergence result of Theorem~\ref{thm:bubble_sort_diffusion}.

\subsection{The finite-dimensional distributions}

We start by establishing weak convergence of the fdds of $\tilde{B}_{n}$. For ease of exposition, we restrict ourselves to weak convergence of two-dimensional distributions of the form $(\tilde{B}_{n} (x_{1}), \tilde{B}_{n} (x_{2}))$, where $x_{1} < x_{2}$ and $x_{1}, x_{2} \not= x^{*}_{\beta}$. We separately study three different cases: Case I with $x^{*}_{\beta} < x_{1} < x_{2}$, Case II with $x_{1} < x^{*}_{\beta} < x_{2}$, and Case III with $x_{1} < x_{2} < x^{*}_{\beta}$.

\subsubsection{Notation}
We introduce some new notation to prove the desired convergence results. 
In line with the notation of the previous section, we define $I_{k} (x)$ as the position of the $k$-th value exceeding $x$ in the underlying sequence of random variables $X_{1}, X_{2}, \dotsc$. We also define $H_{k} (x) = \sum_{i = 1}^{k} \ind{X_{i} > x}$ as the total number of values that exceed $x$ in $X_{1}, \dotsc, X_{k}$.

We let $G_{1} (x_{1}), G_{2} (x_{1}), \dotsc$ denote a collection of independent random variables that follow a geometric distribution with parameter $p_{1} = 1 - F(x_{1})$. The variable $G_{1} (x_{1})$ is the position of the first value exceeding $x_{1}$ in the sequence $X_{1}, X_{2}, \dotsc$, while $G_{1} (x_{1}) + G_{2} (x_{1})$ is the position of the second value exceeding $x_{1}$ in that sequence and so on.

Each value in $X_{1}, X_{2}, \dotsc$ that exceeds $x_{1}$ also exceeds $x_{2}$ with probability $p_{2} = (1 - F(x_{2})) / (1 - F(x_{1}))$. This gives rise to another collection $G_{1} (x_{2}), G_{2} (x_{2}), \dotsc$ of independent random variables having a geometric distribution with parameter $p_{2}$. Thus $G_{1} (x_{2})$ is the position of the first value exceeding $x_{2}$ in the sequence of values exceeding $x_{1}$, while $G_{1} (x_{2}) + G_{2} (x_{2})$ is the position of the second value exceeding $x_{2}$ in the sequence of values exceeding $x_{1}$ and so on.

Given these definitions, we let $S_{k} (x) = \sum_{i = 1}^{k} G_{i} (x)$. With this notation, the position of the $k$-th value exceeding $x_{1}$ in $X_{1}, X_{2}, \dotsc$ is given by
\begin{align}
I_{k} (x_{1}) = \sum_{i = 1}^{k} G_{i} (x_{1}) = S_{k} (x_{1})
\end{align}
and the position of the $k$-th value exceeding $x_{2}$ in $X_{1}, X_{2}, \dotsc$ is given by
\begin{align}
I_{k} (x_{2}) = \sum_{i = 1}^{\sum_{j = 1}^{k} G_{j} (x_{2})} G_{i} (x_{1}) = S_{S_{k} (x_{2})} (x_{1}).
\end{align}

The partial sum $S_{k} (x)$ gives rise to an associated counting process $C_{t} (x)$ defined via $C_{t} (x) = \max \lbrace k \geq 0 \, \vert \, S_{k} (x) \leq t \rbrace$, where we take $t \in [0, \infty)$ and interpret $S_{0} (x) = 0$. The random variable $C_{k} (x_{1})$ simply counts the number of values exceeding $x_{1}$ in $X_{1}, \dotsc, X_{k}$, while $C_{k} (x_{2})$ counts how many of the first $k$ values that exceed $x_{1}$ also exceed $x_{2}$. Therefore we may represent
\begin{align}
H_{k} (x_{1}) = C_{k} (x_{1})
\end{align}
and
\begin{align}
H_{k} (x_{2}) = C_{C_{k} (x_{1})} (x_{2}).
\end{align}
Thus we can represent each of the variables $I_{n} (x_{1})$, $I_{n} (x_{2})$, $H_{n} (x_{1})$, and $H_{n} (x_{2})$ in terms of nested partial sums or counting processes. The convergence properties of these nested versions of partial sums and their associated counting processes form the crux of the proof.

\subsubsection{Proof of weak convergence}
We are now in a position to establish the desired weak convergence results. We start with convergence of the counting processes and their associated partial sums, which we exploit to prove weak convergence of $(\tilde{B}_{n} (x_{1}), \tilde{B}_{n} (x_{2}))$.

With $t \in [0, \infty)$ and $[t]$ being equal to $t$ rounded down to the nearest integer, we may regard $C_{t} (x_{i})$ and $S_{[t]} (x_{i})$ as right-continuous stochastic processes on $[0, \infty)$. In view of \cite[Th.~7.3.2]{whitt2002}, their scaled and centered versions
\begin{align*}
\frac{1}{\sqrt{n}} \big( C_{nt} (x_{i}) - p_{i} nt \big)
\quad \text{ and } \quad
\frac{1}{\sqrt{n}} \big( S_{[nt]} (x_{i}) - \frac{1}{p_{i}} nt \big)
\end{align*}
converge jointly to the processes
\begin{align*}
-\sqrt{1 - p_{i}} W_{p_{i} t}^{(i)}
\quad \text{ and } \quad
\frac{\sqrt{1 - p_{i}}}{p_{i}} W_{t}^{(i)},
\end{align*}
where $W_{t}^{(i)}$ is a standard Brownian motion with $W_{t}^{(1)}$ and $W_{t}^{(2)}$ independent. Combined with the new representations of $I_{n} (x_{1})$, $I_{n} (x_{2})$, $H_{n} (x_{1})$, and $H_{n} (x_{2})$ above, we can use this result to derive weak convergence of $(\tilde{B}_{n} (x_{1}), \tilde{B}_{n} (x_{2}))$ for all three cases.

\emph{Case I.} If $x^{*}_{\beta} < x_{1} < x_{2}$, then \eqref{eq:Bnx} implies that $(\tilde{B}_{n} (x_{1}), \tilde{B}_{n} (x_{2}))$ is asymptotically equivalent to
\begin{align*}
-\bigg( \sqrt{n} \bigg( \frac{1}{n} H_{n} (x_{1}) - (1 - F(x_{1}) \bigg), \sqrt{n} \bigg( \frac{1}{n} H_{n} (x_{2}) - (1 - F(x_{2})) \bigg) \bigg).
\end{align*}
The first element of this random vector equals $\frac{1}{\sqrt{n}} (C_{n} (x_{1}) - p_{1}n)$, while the second equals $\frac{1}{\sqrt{n}} \big( C_{C_{n} (x_{1})} (x_{2}) - p_{1} p_{2} \big)$. We can rewrite the last expression as
\begin{align}
\label{eq:hatB_tsx2}
\frac{1}{\sqrt{n}} \big( C_{n \bar{C}_{n} (x_{1})} (x_{2}) - p_{2} n \bar{C}_{n} (x_{1}) \big) + p_{2} \frac{1}{\sqrt{n}} \big( C_{n} (x_{1}) - p_{1}n \big),
\end{align}
where we define $\bar{C}_{n} (x_{1}) = \frac{1}{n} C_{n} (x_{1})$. Based on these considerations, we conclude that $(\tilde{B}_{n} (x_{1}), \tilde{B}_{n} (x_{2}))$ is asymptotically equivalent to
\begin{align*}
-\bigg(
\frac{1}{\sqrt{n}} (C_{n} (x_{1}) - p_{1}n),
p_{2} \frac{1}{\sqrt{n}} \big( C_{n} (x_{1}) - p_{1}n \big) + \frac{1}{\sqrt{n}} \big( C_{n \bar{C}_{n} (x_{1})} (x_{2}) - p_{2} n \bar{C}_{n} (x_{1}) \big)
\bigg).
\end{align*}
Given the convergence of counting processes and their partial sums as well as the convergence in probability of $\bar{C}_{n} (x_{1})$ to $p_{1}$, we see that this vector converges weakly to
\begingroup
\renewcommand*{\arraystretch}{1.75}
\setlength\arraycolsep{4pt}
\begin{align*}
\begin{bmatrix}
1 & 0 \\
p_{2} & 1
\end{bmatrix}
\begin{bmatrix}
-\sqrt{1 - p_{1}} W_{p_{1}}^{(1)} \\
-\sqrt{1 - p_{2}} W_{p_{2} p_{1}}^{(2)}
\end{bmatrix}
=
-
\begin{bmatrix}
\sqrt{1 - p_{1}}  & 0 \\
\sqrt{1 - p_{1}} p_{2} & \sqrt{1 - p_{2}}
\end{bmatrix}
\begin{bmatrix}
W_{p_{1}}^{(1)} \\
W_{p_{1} p_{2}}^{(2)}
\end{bmatrix}
.
\end{align*}
\endgroup
Since $W^{(1)}$ and $W^{(2)}$ are independent standard Brownian motions, this implies that $(\tilde{B}_{n} (x_{1}), \tilde{B}_{n} (x_{2}))$ converges weakly to a zero-mean normal distribution with covariance matrix
\begingroup
\renewcommand*{\arraystretch}{1.75}
\setlength\arraycolsep{4pt}
\begin{align*}
\begin{bmatrix}
\sqrt{1 - p_{1}}  & 0 \\
\sqrt{1 - p_{1}} p_{2} & \sqrt{1 - p_{2}}
\end{bmatrix}
\begin{bmatrix}
p_{1} & 0 \\
0 & p_{1} p_{2}
\end{bmatrix}
\begin{bmatrix}
\sqrt{1 - p_{1}}  & \sqrt{1 - p_{1}} p_{2} \\
0 & \sqrt{1 - p_{2}}
\end{bmatrix}
.
\end{align*}
\endgroup
Bearing in mind that $p_{1} = 1 - F(x_{1})$ and $p_{1} p_{2} = 1 - F(x_{2})$, some straightforward computations reveal that the last matrix equals
\begingroup
\renewcommand*{\arraystretch}{1.75}
\setlength\arraycolsep{4pt}
\begin{align*}
\begin{bmatrix}
p_{1} (1 - p_{1}) & p_{1} (1 - p_{1}) p_{2} \\
p_{1} (1 - p_{1}) p_{2} & p_{1} p_{2} (1 - p_{1} p_{2})
\end{bmatrix}
=
\begin{bmatrix}
F(x_{1}) (1 - F(x_{1})) & F(x_{1}) (1 - F(x_{2})) \\
F(x_{1}) (1 - F(x_{2})) & F(x_{2}) (1 - F(x_{2}))
\end{bmatrix}
.
\end{align*}
\endgroup
The entries of the last matrix are indeed characterized by the function $c$ specified in \eqref{def:Y_covariance}. Thus $(\tilde{B}_{n} (x_{1}), \tilde{B}_{n} (x_{2}))$ converges weakly to a zero-mean normal distribution whose covariance matrix is specified by \eqref{def:Y_covariance}. Not surprisingly, this is the limiting distribution that one obtains in a Kolmogorov--Smirnov setting with full sorting.

\emph{Case II.} If $x_{1} < x^{*}_{\beta} < x_{2}$, then \eqref{eq:Bnx} implies that $(\tilde{B}_{n} (x_{1}), \tilde{B}_{n} (x_{2}))$ is asymptotically equivalent to
\begin{align*}
\bigg( \sqrt{n} \bigg( \frac{1}{n} I_{\beta n} (x_{1}) - \beta - \beta \frac{F(x_{1})}{1 - F(x_{1})} \bigg), - \sqrt{n} \bigg( \frac{1}{n} H_{n} (x_{2}) - (1 - F(x_{2})) \bigg) \bigg).
\end{align*}
The first element of this random vector equals $\frac{1}{\sqrt{n}} (S_{[n \beta]} (x_{1}) - \frac{1}{p_{1}} n \beta)$, while the second element can be rewritten as \eqref{eq:hatB_tsx2}. Thus $(\tilde{B}_{n} (x_{1}), \tilde{B}_{n} (x_{2}))$ is asymptotically equivalent to
\begin{align*}
\bigg(
\frac{1}{\sqrt{n}} (S_{[n \beta]} (x_{1}) - \frac{1}{p_{1}} n \beta),
- p_{2} \frac{1}{\sqrt{n}} \big( C_{n} (x_{1}) - p_{1}n \big) + \frac{1}{\sqrt{n}} \big( C_{n \bar{C}_{n} (x_{1})} (x_{2}) - p_{2} n \bar{C}_{n} (x_{1}) \big)
\bigg).
\end{align*}
Again relying on the convergence of counting processes and their partial sums as well as the convergence in probability of $\bar{C}_{n} (x_{1})$ to $p_{1}$, we see that this vector converges weakly to
\begingroup
\renewcommand*{\arraystretch}{1.75}
\setlength\arraycolsep{4pt}
\begin{align*}
\begin{bmatrix}
1 & 0 & 0 \\
0 & -p_{2} & 1
\end{bmatrix}
\begin{bmatrix}
\frac{\sqrt{1 - p_{1}}}{p_{1}} W_{\beta}^{(1)} \\
-\sqrt{1 - p_{1}} W_{p_{1}}^{(1)} \\
-\sqrt{1 - p_{2}} W_{p_{2} p_{1}}^{(2)}
\end{bmatrix}
=
\begin{bmatrix}
\sqrt{1 - p_{1}}{p_{1}} & 0 & 0 \\
0 & \sqrt{1 - p_{1}} p_{2} & - \sqrt{1 - p_{2}}
\end{bmatrix}
\begin{bmatrix}
W_{\beta}^{(1)} \\
W_{p_{1}}^{(1)} \\
W_{p_{2} p_{1}}^{(2)}
\end{bmatrix}
.
\end{align*}
\endgroup
Since $W^{(1)}$ and $W^{(2)}$ are independent standard Brownian motions and $\beta < p_{1}$, this implies that $(\tilde{B}_{n} (x_{1}), \tilde{B}_{n} (x_{2}))$ converges weakly to a zero-mean normal distribution with covariance matrix
\begin{align*}
\begin{bmatrix}
\sqrt{1 - p_{1}}{p_{1}} & 0 & 0 \\
0 & \sqrt{1 - p_{1}} p_{2} & - \sqrt{1 - p_{2}}
\end{bmatrix}
\begin{bmatrix}
\beta & \beta & 0 \\
\beta & p_{1} & 0 \\
0 & 0 & p_{1} p_{2}
\end{bmatrix}
\begin{bmatrix}
\sqrt{1 - p_{1}}{p_{1}} & 0 \\
0 & \sqrt{1 - p_{1}} p_{2} \\
0 & - \sqrt{1 - p_{2}}
\end{bmatrix}
,
\end{align*}
which equals
\begingroup
\renewcommand*{\arraystretch}{2.0}
\setlength\arraycolsep{5pt}
\begin{align*}
\begin{bmatrix}
\beta \frac{1 - p_{1}}{p_{1}^{2}} & \beta \frac{1 - p_{1}}{p_{1}} p_{2} \\
\beta \frac{1 - p_{1}}{p_{1}} p_{2} & p_{1} p_{2} (1 - p_{1} p_{2})
\end{bmatrix}
=
\begin{bmatrix}
\beta \frac{F(x_{1})}{(1 - F(x_{1}))^{2}} & \beta \frac{F(x_{1})}{(1 - F(x_{1}))^{2}} (1 - F(x_{2})) \\
\beta \frac{F(x_{1})}{(1 - F(x_{1}))^{2}} (1 - F(x_{2})) & F(x_{2}) (1 - F(x_{2}))
\end{bmatrix}
.
\end{align*}
\endgroup
We conclude that $(\tilde{B}_{n} (x_{1}), \tilde{B}_{n} (x_{2}))$ converges weakly to a zero-mean normal distribution whose covariance matrix is specified by \eqref{def:Y_covariance}.

\emph{Case III.} If $x_{1} < x_{2} < x^{*}_{\beta}$, then \eqref{eq:Bnx} implies that $(\tilde{B}_{n} (x_{1}), \tilde{B}_{n} (x_{2}))$ is asymptotically equivalent to
\begin{align*}
\bigg( \sqrt{n} \bigg( \frac{1}{n} I_{\beta n} (x_{1}) - \beta - \beta \frac{F(x_{1})}{1 - F(x_{1})} \bigg), \sqrt{n} \bigg( \frac{1}{n} I_{\beta n} (x_{2}) - \beta - \beta \frac{F(x_{2})}{1 - F(x_{2})} \bigg) \bigg).
\end{align*}
The first element of this random vector equals $\frac{1}{\sqrt{n}} (S_{[n \beta]} (x_{1}) - \frac{1}{p_{1}} n \beta)$, while the second element equals $\frac{1}{\sqrt{n}} (S_{S_{[n \beta]} (x_{2})} (x_{1}) - \frac{1}{p_{1} p_{2}} n \beta)$. We can rewrite the last expression as
\begin{align*}
\frac{1}{\sqrt{n}} \big( S_{n \bar{S}_{[n \beta]} (x_{2})} (x_{1}) - \frac{1}{p_{1}} n \bar{S}_{[n \beta]} (x_{2}) \big) + \frac{1}{p_{1}} \frac{1}{\sqrt{n}} (S_{[n \beta]} (x_{2}) - \frac{1}{p_{2}} n \beta),
\end{align*}
where we define $\bar{S}_{[n \beta]} (x_{2}) = \frac{1}{n} S_{[n \beta]} (x_{2})$. 
Thus $(\tilde{B}_{n} (x_{1}), \tilde{B}_{n} (x_{2}))$ is asymptotically equivalent to
\begin{align*}
\bigg(
\frac{1}{\sqrt{n}} (S_{[n \beta]} (x_{1}) - \frac{1}{p_{1}} n \beta),
\frac{1}{\sqrt{n}} \big( S_{n \bar{S}_{[n \beta]} (x_{2})} (x_{1}) - \frac{1}{p_{1}} n \bar{S}_{[n \beta]} (x_{2}) \big) + \frac{1}{p_{1}} \frac{1}{\sqrt{n}} (S_{[n \beta]} (x_{2}) - \frac{1}{p_{2}} n \beta)
\bigg).
\end{align*}
Taking into account that $\bar{S}_{[n \beta]} (x_{2})$ converges in probability to $\beta / p_{2}$, we see that this vector converges weakly to
\begingroup
\renewcommand*{\arraystretch}{1.75}
\setlength\arraycolsep{4pt}
\begin{align*}
\begin{bmatrix}
1 & 0 & 0 \\
0 & 1 & \frac{1}{p_{1}}
\end{bmatrix}
\begin{bmatrix}
\frac{\sqrt{1 - p_{1}}}{p_{1}} W_{\beta}^{(1)} \\
\frac{\sqrt{1 - p_{1}}}{p_{1}} W_{\beta / p_{2}}^{(1)} \\
\frac{\sqrt{1 - p_{2}}}{p_{2}} W_{\beta}^{(2)}
\end{bmatrix}
=
\begin{bmatrix}
\frac{\sqrt{1 - p_{1}}}{p_{1}} & 0 & 0 \\
0 & \frac{\sqrt{1 - p_{1}}}{p_{1}} & \frac{\sqrt{1 - p_{2}}}{p_{1} p_{2}}
\end{bmatrix}
\begin{bmatrix}
W_{\beta}^{(1)} \\
W_{\beta / p_{2}}^{(1)} \\
W_{\beta}^{(2)}
\end{bmatrix}
.
\end{align*}
\endgroup
This implies that $(\tilde{B}_{n} (x_{1}), \tilde{B}_{n} (x_{2}))$ converges weakly to a zero-mean normal distribution with covariance matrix
\begingroup
\renewcommand*{\arraystretch}{1.75}
\setlength\arraycolsep{4pt}
\begin{align*}
\begin{bmatrix}
\frac{\sqrt{1 - p_{1}}}{p_{1}} & 0 & 0 \\
0 & \frac{\sqrt{1 - p_{1}}}{p_{1}} & \frac{\sqrt{1 - p_{2}}}{p_{1} p_{2}}
\end{bmatrix}
\begin{bmatrix}
\beta & \beta & 0 \\
\beta & \frac{\beta}{p_{2}} & 0 \\
0 & 0 & \beta
\end{bmatrix}
\begin{bmatrix}
\frac{\sqrt{1 - p_{1}}}{p_{1}} & 0 \\
0 & \frac{\sqrt{1 - p_{1}}}{p_{1}} \\
0 & \frac{\sqrt{1 - p_{2}}}{p_{1} p_{2}}
\end{bmatrix}
,
\end{align*}
\endgroup
which equals
\begingroup
\renewcommand*{\arraystretch}{1.75}
\setlength\arraycolsep{4pt}
\begin{align*}
\begin{bmatrix}
\beta \frac{1 - p_{1}}{p_{1}^{2}} & \beta \frac{1 - p_{1}}{p_{1}^{2}} \\
\beta \frac{1 - p_{1}}{p_{1}^{2}} & \beta \frac{1 - p_{1} p_{2}}{p_{1}^{2} p_{2}^{2}}
\end{bmatrix}
=
\begin{bmatrix}
\beta \frac{F(x_{1})}{(1 - f(x_{1}))^{2}} & \beta \frac{F(x_{1})}{(1 - f(x_{1}))^{2}} \\
\beta \frac{F(x_{1})}{(1 - f(x_{1}))^{2}} & \beta \frac{F(x_{2})}{(1 - f(x_{2}))^{2}}
\end{bmatrix}
.
\end{align*}
\endgroup
We conclude that $(\tilde{B}_{n} (x_{1}), \tilde{B}_{n} (x_{2}))$ converges weakly to a zero-mean normal distribution whose covariance matrix is specified by \eqref{def:Y_covariance}.\\

The previous arguments establish the convergence of the two-dimensional distributions of $\tilde{B}_{n}$ on a dense subset of the real line. The computations already reveal that the limit is Gaussian and provide the corresponding covariance matrix. Convergence of the higher-dimensional distribution can be derived along similar lines.

\subsection{Tightness}
We now turn to tightness of the process $\tilde{B}_{n}$ in the M1 topology. For completeness, we first present characterizations of tightness and then use these to establish the M1 tightness of $\tilde{B}_{n}$. Further background for tightness and the M1 topology can be found in \cite[Ch.~11--13]{whitt2002}.

A sequence of stochastic processes $Z_{n}$ on $\mathbb{R}$ is C-tight if for every interval $[-\tau, \tau]$ the following two conditions are met. First, for every $\epsilon > 0$ there exists a positive constant $c$ such that $\mathbb{P} (\abs{Z_{n} (-\tau)} > c) < \epsilon$ for all $n$. Second, for every $\epsilon > 0$ and $\eta > 0$ there exist $\delta > 0$ and $N$ such that
\begin{align*}
\mathbb{P} \Big( u_{\tau} (Z_{n}, \delta) > \eta \Big) < \epsilon
\end{align*}
for all $n \geq N$. The oscillation function $u_{\tau}$ is defined by taking
\begin{align*}
u_{\tau} (z, x, \delta) = \sup \lbrace \abs{z(x_{1}) - z(x_{2})} \, : \, -\tau \vee (x - \delta) \leq x_{1} < x_{2} \leq (x + \delta) \wedge \tau \rbrace
\end{align*}
and
\begin{align*}
u_{\tau} (z, \delta) = \sup_{-\tau \leq x \leq \tau} u_{\tau} (z, x, \delta).
\end{align*}
The process $Z_{n}$ converges weakly to a continuous stochastic process $Z$ if and only if its fdds converge weakly to the fdds of $Z$ and $Z_{n}$ is C-tight.

The process $\tilde{B}_{n}$ clearly does not converge to a continuous stochastic process: its limit exhibits a jump at the point $x^{*}_{\beta}$.  {This is apparent in Figure \ref{fig:discontinuity_uniform} as well as in the form of the covariance structure in \eqref{def:Y_covariance}.} In view of this, the M1 topology is the natural choice for establishing weak convergence of $\tilde{B}_{n}$.

A sequence of stochastic processes $Z_{n}$ on $\mathbb{R}$ converges to $Z$ in the M1 topology if its fdds converge weakly to the fdds of $Z$ on a dense subset of $\mathbb{R}$ and $Z_{n}$ is tight in the M1 topology. The sequence $Z_{n}$ is tight in the M1 topology if for every element $\tau$ in a positive, unbounded sequence $\lbrace \tau_{k} \rbrace_{k=1}^{\infty}$ the following two conditions are met. First, for every $\epsilon > 0$ there exists a positive constant $c$ such that $\mathbb{P} (\sup \lbrace \abs{Z_{n} (x)} \, : \, -\tau \leq x \leq \tau \rbrace > c) < \epsilon$ for all $n$. Second, for every $\epsilon > 0$ and $\eta > 0$ there exist $\delta > 0$ and $N$ such that
\begin{align*}
\mathbb{P} \Big( w_{\tau} (Z_{n}, \delta) > \eta \Big) < \epsilon
\end{align*}
for all $n \geq N$. The oscillation function $w_{\tau}$ is defined as follows. We interpret $[a, b]$ as the interval $[a \wedge b, a \vee b]$ and let $\abs{c - [a, b]}$ denote the distance between a point $c$ and the interval $[a, b]$. We then take
\begin{align*}
v_{\tau} (z, x, \delta) = \sup \lbrace \abs{z(x_{2}) - [z(x_{1}), z(x_{3})]} \, : \, -\tau \vee (x - \delta) \leq x_{1} < x_{2} < x_{3} \leq (x + \delta) \wedge \tau \rbrace,
\end{align*}
and
\begin{align*}
v_{\tau} (z, \delta) = \sup_{-\tau \leq x \leq \tau} v_{\tau} (z, x, \delta),
\end{align*}
which are used to define the oscillation function
\begin{align*}
w_{\tau} (z, \delta) = u_{\tau} (z, -\tau, \delta) \vee v_{\tau} (z, \delta) \vee u_{\tau} (z, \tau, \delta).
\end{align*}
We use these characterizations of tightness to prove that $\tilde{B}_{n}$ is tight in the M1 topology.

To establish tightness of $\tilde{B}_{n}$ in the M1 topology, we fix $\tau > 1 + |x^{*}_{\beta}|$ as well as $\epsilon > 0$ and $\eta > 0$. As a first step we note that the processes
\begin{align*}
\tilde{I}_{n} (x) = \sqrt{n} \Big( \frac{1}{n} I_{1 + \beta n} (x) - \beta \frac{F(x)}{1 - F(x)} \Big)
\end{align*}
and
\begin{align*}
\tilde{H}_{n} (x) = \sqrt{n} \Big( \frac{1}{n} H_{n} (x) - (1 - F(x)) \Big)
\end{align*}
are both C-tight as a consequence of \cite[Th.~11.6.5]{whitt2002}. This implies that there exist $\delta > 0$ and $N$ such that the next three statements are true. First, it implies that $\mathbb{P} (\sup \lbrace \vert \tilde{B}_{n} (x) \vert \, : \, -\tau \leq x \leq \tau \rbrace > c) < \epsilon$ for a sufficiently large constant, since $|\tilde{B}_{n} (x)| \leq |\tilde{I}_{n} (x)| + |\tilde{H}_{n} (x)|$. Second, it implies that $\mathbb{P} (u_{\tau} (\tilde{B}_{n}, -\tau, \delta) > \eta) < \epsilon / 3$ and $\mathbb{P} (u_{\tau} (\tilde{B}_{n}, \tau, \delta) > \eta) < \epsilon / 3$. Third, it implies that
\begin{align*}
\frac{1}{n} I_{1 + \beta n} (x_{1}) 
&= \frac{1}{n} I_{1 + \beta n} (x_{2}) + \Big( \beta \frac{F(x_{1})}{1 - F(x_{1})} - \beta \frac{F(x_{2})}{1 - F(x_{2})} \Big) + o \Big( \frac{\eta}{\sqrt{n}} \Big) \\
&= \frac{1}{n} I_{1 + \beta n} (x_{2}) + \Big( \frac{\beta}{1 - F(x_{1})} - \frac{\beta}{1 - F(x_{2})} \Big) + o \Big( \frac{\eta}{\sqrt{n}} \Big)
\end{align*}
and
\begin{align*}
\frac{1}{n} H_{n} (x_{1}) = \frac{1}{n} H_{n} (x_{2}) + (F(x_{1}) - F(x_{2})) + o \Big( \frac{\eta}{\sqrt{n}} \Big)
\end{align*}
for all $\abs{x_{1} - x_{2}} < \delta$ on a set $\Omega^{*}$ with probability at least $1 - \epsilon / 3$ for all $n \geq N$. Here each $o (c)$ indicates a term whose absolute value is at most $c > 0$.

We now analyze four separate cases for $x_{1} < x_{2} < x_{3}$ on the set $\Omega^{*}$. We show that in each case
\begin{align}
\label{eq:M1inequalities}
\tilde{B}_{n} (x_{1}) + o(2 \eta) \leq \tilde{B}_{n} (x_{2}) \leq \tilde{B}_{n} (x_{3}) + o(2 \eta)
\end{align}
or
\begin{align}
\label{eq:M1inequalities_reversed}
\tilde{B}_{n} (x_{1}) + o(2 \eta) \geq \tilde{B}_{n} (x_{2}) \geq \tilde{B}_{n} (x_{3}) + o(2 \eta),
\end{align}
implying that $\vert \tilde{B}_{n} (x_{2}) - [\tilde{B}_{n} (x_{1}), \tilde{B}_{n} (x_{3})] \vert \leq 2 \eta$. The M1 tightness of $\tilde{B}_{n}$ is a direct consequence of this.

Case I: $\frac{1}{n} H_{n} (x_{3}) > \beta$. In this case $\frac{1}{n} H_{n} (x_{1}) > \beta$ and $\frac{1}{n} H_{n} (x_{2}) > \beta$, so
\begin{align*}
B_{n} (x_{1}) = B_{n} (x_{2}) + \Big( \frac{\beta}{1 - F(x_{1})} - \frac{\beta}{1 - F(x_{2})} \Big) + o \Big( \frac{\eta}{\sqrt{n}} \Big)
\end{align*}
and
\begin{align*}
B_{n} (x_{2}) = B_{n} (x_{3}) + \Big( \frac{\beta}{1 - F(x_{2})} - \frac{\beta}{1 - F(x_{3})} \Big) + o \Big( \frac{\eta}{\sqrt{n}} \Big).
\end{align*}
This yields
\begin{align}
\label{eq:Bn_tightness_inequality_case_1_1}
\tilde{B}_{n} (x_{1}) = \tilde{B}_{n} (x_{2}) + \sqrt{n} \Big( \frac{\beta}{1 - F(x_{1})} - \frac{\beta}{1 - F(x_{2})} + B(x_{2}) - B(x_{1}) \Big) + o(\eta).
\end{align}
and
\begin{align*}
\tilde{B}_{n} (x_{2}) = \tilde{B}_{n} (x_{3}) + \sqrt{n} \Big( \frac{\beta}{1 - F(x_{2})} - \frac{\beta}{1 - F(x_{3})} + B(x_{3}) - B(x_{2}) \Big) + o(\eta).
\end{align*}
Now note that
\begin{align*}
& \frac{\beta}{1 - F(x_{i})} - \frac{\beta}{1 - F(x_{i+1})} + B(x_{i+1}) - B(x_{i}) \\
& {} = \min \Big\lbrace 0, F(x_{i+1}) - \beta \frac{F(x_{i+1})}{1 - F(x_{i+1})} \Big\rbrace - \min \Big\lbrace 0, F(x_{i}) - \beta \frac{F(x_{i})}{1 - F(x_{i})} \Big\rbrace \leq 0,
\end{align*}
since $\min \lbrace 0, F(x) - \beta \frac{F(x)}{1 - F(x)} \rbrace$ is nonincreasing. Therefore \eqref{eq:M1inequalities} holds in this case.

Case II: $\frac{1}{n} H_{n} (x_{3}) \leq \beta$ and $\frac{1}{n} H_{n} (x_{2}) > \beta$. In this case $\frac{1}{n} H_{n} (x_{1}) > \beta$, so \eqref{eq:Bn_tightness_inequality_case_1_1} is satisfied. Based on the equality
\begin{align*}
1 + \beta - \frac{1}{n} I_{1 + \beta n} (x_{2}) - \frac{1}{n} H_{n} (x_{3}) 
\geq 1 + \beta - \frac{1}{n} I_{1 + \beta n} (x_{2}) - \frac{1}{n} H_{n} (x_{2}) 
\geq 0
\end{align*}
as well as analogous arguments as in the previous case, we see that
\begin{align*}
\tilde{B}_{n} (x_{3}) \geq \tilde{B}_{n} (x_{2}) 
&+ \sqrt{n} \max \Big\lbrace F(x_{3}) - (1 - \beta), \frac{\beta}{1 - F(x_{3})} - 1 \Big\rbrace \\
&{} - \sqrt{n} \max \Big\lbrace F(x_{2}) - (1 - \beta), \frac{\beta}{1 - F(x_{2})} - 1 \Big\rbrace + o(2 \eta).
\end{align*}
Since $\max \lbrace F(x) - (1 - \beta), \frac{\beta}{1 - F(x)} - 1 \rbrace$ is nondecreasing, we conclude that $\tilde{B}_{n} (x_{3}) + o(2 \eta) \geq \tilde{B}_{n} (x_{2})$. Therefore \eqref{eq:M1inequalities} holds in this case.

Case III: $\frac{1}{n} H_{n} (x_{2}) \leq \beta$ and $\frac{1}{n} H_{n} (x_{1}) > \beta$. In this case $1 + \beta \leq \frac{1}{n} I_{1 + \beta n} (x_{2}) + \frac{1}{n} H_{n} (x_{1})$, which leads to
\begin{align*}
\tilde{B}_{n} (x_{1}) \geq \tilde{B}_{n} (x_{2}) 
&- \sqrt{n} \max \Big\lbrace F(x_{3}) - (1 - \beta), \frac{\beta}{1 - F(x_{3})} - 1 \Big\rbrace \\
&{} + \sqrt{n} \max \Big\lbrace F(x_{2}) - (1 - \beta), \frac{\beta}{1 - F(x_{2})} - 1 \Big\rbrace + o(2 \eta).
\end{align*}
Since $\max \lbrace F(x) - (1 - \beta), \frac{\beta}{1 - F(x)} - 1 \rbrace$ is nondecreasing, it follows that $\tilde{B}_{n} (x_{1}) + o(2 \eta) \geq \tilde{B}_{n} (x_{2})$. We also note that both $\frac{1}{n} H_{n} (x_{2}) \leq \beta$ and $\frac{1}{n} H_{n} (x_{3}) \leq \beta$, so
\begin{align}
\label{eq:Bn_tightness_inequality_case_3_2}
\tilde{B}_{n} (x_{2}) 
&= \tilde{B}_{n} (x_{3}) + \sqrt{n} \Big( \min \Big\lbrace \beta \frac{F(x_{2})}{1 - F(x_{2})}, 0 \Big\rbrace - \min \Big\lbrace \beta \frac{F(x_{1})}{1 - F(x_{1})}, 0 \Big\rbrace \Big) + o(\eta).
\end{align}
Since $\min \lbrace \beta \frac{F(x)}{1 - F(x)}, 0 \rbrace$ is nondecreasing, this yields $\tilde{B}_{n} (x_{2}) \geq \tilde{B}_{n} (x_{3}) + o(\eta)$. Therefore \eqref{eq:M1inequalities_reversed} holds in this case.

Case IV: $\frac{1}{n} H_{n} (x_{1}) \leq \beta$. In this case $\frac{1}{n} H_{n} (x_{3}) \leq \frac{1}{n} H_{n} (x_{2}) \leq \frac{1}{n} H_{n} (x_{1}) \leq \beta$, implying that \eqref{eq:Bn_tightness_inequality_case_3_2} is valid and remains so if we replace $x_{2}$ by $x_{1}$ and $x_{3}$ by $x_{2}$. Therefore \eqref{eq:M1inequalities_reversed} holds in this case.

\section{Analyzing the bubble sort statistic}
\label{sec:analyzing}

In this section we prove Theorem~\ref{thm:gKS_convergence} and Lemma~\ref{lem:supremum_cdf}. 
The result of Theorem \ref{thm:bubble_sort_diffusion} combined with the fact that weak convergence of a process in the M1 topology implies weak convergence of its supremum (cf.\ \cite[Th.~13.4.1]{whitt2002}), implies that $\hat{D}^\beta_n = \sup_{x\in\R} \sqrt{n}(\hat B^\beta_n(x) - B^\beta_0(x))$ converges in distribution to $\sup_{x\in\R} Y^\beta_0(x)$. It remains to show that $\sup_{x\in\R} Y_0^\beta(x)$ equals ${D}^\beta$ in distribution (Theorem~\ref{thm:gKS_convergence}), and to derive an expression for the distribution of ${D}^\beta$ (Lemma~\ref{lem:supremum_cdf}).


We prove Theorem~\ref{thm:gKS_convergence} in two steps. First, conditional on the jump of $Y_0^\beta$ in $x^{*}_{\beta}$, we relate the distribution of $\sup_{x\in\R} Y_0^\beta(x)$ to the distribution of generalized Brownian bridges. Based on this, we then argue that $\sup_{x\in\R} Y_0^\beta(x)$ indeed follows a generalized Kolmogorov distribution.

Consider a right-continuous, centred Gaussian process $\{\check Y^\beta(t) : t\in[0,1]\}$ with covariance function
\begin{align*}
\cov(\check Y^\beta(t_1),\check Y^\beta(t_2)) = 
\begin{cases}
\frac{\beta t_1}{(1-t_1)^2} & \text{ if } t_1 \leq t_2 < 1 - \beta, \\
\frac{\beta t_1(1-t_2)}{(1-t_1)^2} & \text{ if } t_1 < 1 - \beta < t_2, \\
t_1(1-t_2) & \text{ if }  1 - \beta < t_1 \leq t_2.
\end{cases}
\end{align*}
Clearly $Y^\beta_0(t) = \check Y^\beta(F_0(t))$ and $\sup_t |\check Y^\beta(t)| = \sup_x |Y^\beta_0(x)|$, since $F_0$ is continuous. We define the function $\omega(t) = \beta t/(1-t)^2$, which is strictly increasing for $0 < t < 1$. Additionally, we define the two stochastic processes $\{W_1(t) : t\in[0,\tfrac{1-\beta}{\beta}]\}$ and $\{W_2(t) : t\in[0, \beta]\}$, where
\begin{align*}
W_1(t) & = \check Y^\beta(\omega^{-1}(t)), \ t\in[0,\tfrac{1-\beta}{\beta}) \quad \text{and} \quad W_1(\tfrac{1-\beta}{\beta}) = \lim_{t\to(1-\beta)^-} \check Y^\beta(\omega^{-1}(t)), \\
W_2(t) & = \check Y^\beta(1-t), \ t\in[0,\beta] \quad \text{and} \quad W_2(\beta) = \lim_{t\to(1-\beta)^+} \check Y^\beta(t).
\end{align*}
These processes can be thought of as the (time-scaled and reversed) left and right parts of $\check Y^\beta$, viewed relative to the time point $t=1-\beta$. Similar to before, we have
\[
\sup_t |\check Y^\beta(t)| = \sup_t |W_1(t)| ~ \vee ~\sup_t |W_2(t)|
\]
and
\begin{align*}
\cov(W_1(t_1), W_1(t_2)) & = t_1 &&\text{for} \quad 0\leq t_1\leq t_2\leq \tfrac{1-\beta}{\beta}, \\
\cov(W_1(t_1), W_2(t_2)) & = t_1t_2 &&\text{for} \quad t_1\in[0,\tfrac{1-\beta}{\beta}], t_2\in[0,\beta], \\
\cov(W_2(t_1), W_2(t_2)) & = t_1(1-t_2) &&\text{for} \quad 0 \leq t_1 \leq t_2 \leq \beta.
\end{align*}

We proceed to study the distributions of $W_1$ and $W_2$ conditional on $W_2(\beta)$. Since $W_1$ and $W_2$ are jointly Gaussian, their conditional distribution is also jointly Gaussian. To determine this conditional distribution, it suffices to determine the conditional means, covariances, and cross-correlations.

First, we focus on the cross-correlation. For any $t_1\in[0,\tfrac{1-\beta}{\beta}]$ and $t_2\in[0,\beta]$ we have
\begin{align*}
& \cov(W_1(t_1), W_2(t_2) \mid W_2(\beta)) \\
& = \cov(W_1(t_1), W_2(t_2)) - \frac{\cov(W_1(t_1), W_2(\beta)) \cdot \cov(W_2(t_2), W_2(\beta))}{\var(W_2(\beta))} \\
& = t_1t_2 - \frac{t_1 \beta \cdot t_2(1-\beta)}{\beta(1-\beta)} \\
& = 0,
\end{align*}
implying that $W_1$ and $W_2$ are independent conditional on $W_2(\beta)$.

Second, we study the conditional distribution of $W_1$. For any $t\in[0,\tfrac{1-\beta}{\beta}]$ we have
\begin{align*}
\e(W_1(t) \mid W_2(\beta)) & = \e W_1(t) - \frac{\cov(W_1(t_1), W_2(\beta))}{\var(W_2(\beta))} \cdot \left(W_2(\beta) - \e W_2(\beta)\right) \\
& = \frac{t \beta}{\beta(1-\beta)} \cdot W_2(\beta) \\
& = \frac{t}{\frac{1-\beta}{\beta}} \cdot \frac{W_2(\beta)}{\beta}.
\end{align*}
For any $0\leq t_1\leq t_2\leq \tfrac{1-\beta}{\beta}$ we have
\begin{align*}
& \cov(W_1(t_1), W_1(t_2) \mid W_2(\beta)) \\
& = \cov(W_1(t_1), W_1(t_2)) - \frac{\cov(W_1(t_1), W_2(\beta)) \cdot \cov(W_1(t_2), W_2(\beta))}{\var(W_2(\beta))} \\
& = t_1 - \frac{t_1 \beta \cdot t_2 \beta}{\beta(1-\beta)} \\
& = \frac{t_1\cdot(\tfrac{1-\beta}{\beta}-t_2)}{\frac{1-\beta}{\beta}}.
\end{align*}
Similarly we obtain the conditional mean and covariance of $W_2$. For any $t\in[0,\beta]$ and $0\leq t_1\leq t_2\leq \beta$ we have
\begin{align*}
\e(W_2(t) \mid W_2(\beta)) = \frac{t}{\beta} \cdot W_2(\beta)
\end{align*}
and
\begin{align*}
\cov(W_2(t_1), W_1(t_2) \mid W_2(\beta)) = \frac{t_1(\beta-t_2)}{\beta}.
\end{align*}

The previous arguments demonstrate that, conditional on $W_2(\beta)$, the processes $W_1$ and $W_2$ are independent. Using the formulas for the mean and covariance of a generalized Brownian bridge. We conclude that
\begin{equation}
\label{eq:law-w1-w2-different}
\big(W_1\mid W_2(\beta)\big) \sim {\rm BB}(\tfrac{1-\beta}{\beta},\tfrac{W_2(\beta)}{\beta}) \quad \text{and} \quad \big(W_2 \mid W_2(\beta)\big) \sim {\rm BB}(\beta, W_2(\beta)).
\end{equation}
Further, since $W_2(\beta)$ is a normally distributed zero mean random variable with variance $\beta(1-\beta)$, it is equal in distribution to 
$\sqrt{\beta(1-\beta)}Z$, where $Z$ is a standard normal random variable. Hence the representation in \eqref{eq:law-w1-w2-different} is equivalent to \eqref{eq:law-w1-w2} which is used to construct $D^\beta$ as in \eqref{eq:D-beta-prob-def}. This concludes the proof of Theorem~\ref{thm:gKS_convergence}.

We close this section with a derivation of the cdf of $D^\beta$ as defined in  \eqref{eq:gen-kolg-dist}.

\begin{proof}[Proof of Lemma~\ref{lem:supremum_cdf}]
Based on the conditional independence between $W_1$ and $W_2$ as well as the definition of $\Psi(x;T,a)$ of \eqref{eq:psi-x-t-a-def}, we obtain

\begin{align*}
\p\Big(D^\beta \leq x\Big) & = \e\Big[\p\big(\sup_{t \in [0, (1-\beta)/\beta]} |W_1(t)| \leq x, \sup_{t \in [0, \beta]} |W_2(t)| \leq x \mid Z\big)\Big] \\
& = \e\Big[\p\big(\sup_{t \in [0, (1-\beta)/\beta]} |W_1(t)| \leq x \mid Z \big)~\p\big(\sup_{t \in [0, \beta]} |W_2(t)| \leq x \mid Z\big)\Big] \\
& = \int_{-\infty}^\infty \Psi \big(x;\tfrac{1-\beta}{\beta}, \sqrt{\tfrac{1-\beta}{\beta}}\cdot z \big) \Psi \big(x;\beta, \sqrt{\beta(1-\beta)}\cdot z\big)\phi(z) \, {\rm d}z.
\end{align*}

The integrand in the last line is non-zero if and only if $x > \sqrt{(1-\beta)/\beta}\cdot |z|$ and $x > \sqrt{(1-\beta)\beta}\cdot |z|$. Finally, using that $\Psi(x;T,a) = \Psi(x;T,-a)$ for all $a\in\R$, we obtain the desired result.
\end{proof}

\section{Numerical Examples}
\label{sec:numerical}

We now present two scenarios where the bubble sort based procedure of Section~\ref{sec:gen-procedure} is useful with the purpose of showing the potential value of nonparametric testing with partial sorting. The first scenario deals with detecting partial sorting of tabular data, and the second deals with detecting service time dependent scheduling policies in queues. Both of these examples illustrate that there are cases where partial sorting based procedures can be used to detect lack of independence in a better manner than KS or~WW.

In fact, in any situation where one considers the hypotheses \eqref{eq:basic-hypothesis}, it is possible to use partial sorting ($\beta < 1$) in place of the standard $\beta = 1$ KS test, or the WW test. However, in our experimentation with standard cases involving simulated data from auto-regressive sequences, or data from hidden Markov sequences, we did not always observe a clear benefit in partial sorting. This may be due to the fact that the number of observations falling on the empirical bubble sort curve \eqref{eq:emp-bubble-sort-curve} is typically significantly smaller than those falling on the ecdf of the data. Nevertheless, as the following two examples show, there certainly exist cases where partial sorting is beneficial.

In both examples below, we assume that the distribution $F_0$ of \eqref{eq:basic-hypothesis} is known and is not the subject of testing. Instead our focus is on testing for lack of independence which is detected when $H_0$ is rejected. Further, in both examples we consider a grid of sorting levels $\beta \in (0,1]$. As we see, the performance of the test is heavily influenced by the choice of $\beta$. Policies to optimally choose $\beta$, or perhaps to carry out more involved procedures based on multiple $\beta$ values, are not within the scope of the current paper. Our purpose here is simply to illustrate the potential value of nonparametric testing with partial sorting. 

In our experimentation with both examples, data is generated via simulations in scenarios where $H_0$ does not hold. We then quantify performance by fixing $\alpha = 0.1$ and estimating the probability of rejecting $H_0$ via repeated simulation experiments. The results are then plotted for $\beta \in (0,1]$, where we keep in mind that $\beta = 1$ is the KS test. We also summarize results from the WW test on the same data for comparison.

\subsection*{Example 1: Detection of magnitude based sorting of a hidden column in tabular data}

Consider the common setting of numerical tabular data, e.g.\ a data frame or a spreadsheet, where different variables constitute different columns and observations are in rows. With such data it is very common to sort the observations by a given variable/column. This then modifies the order of other columns accordingly.

\begin{figure}
\centering
\includegraphics[width=10cm]{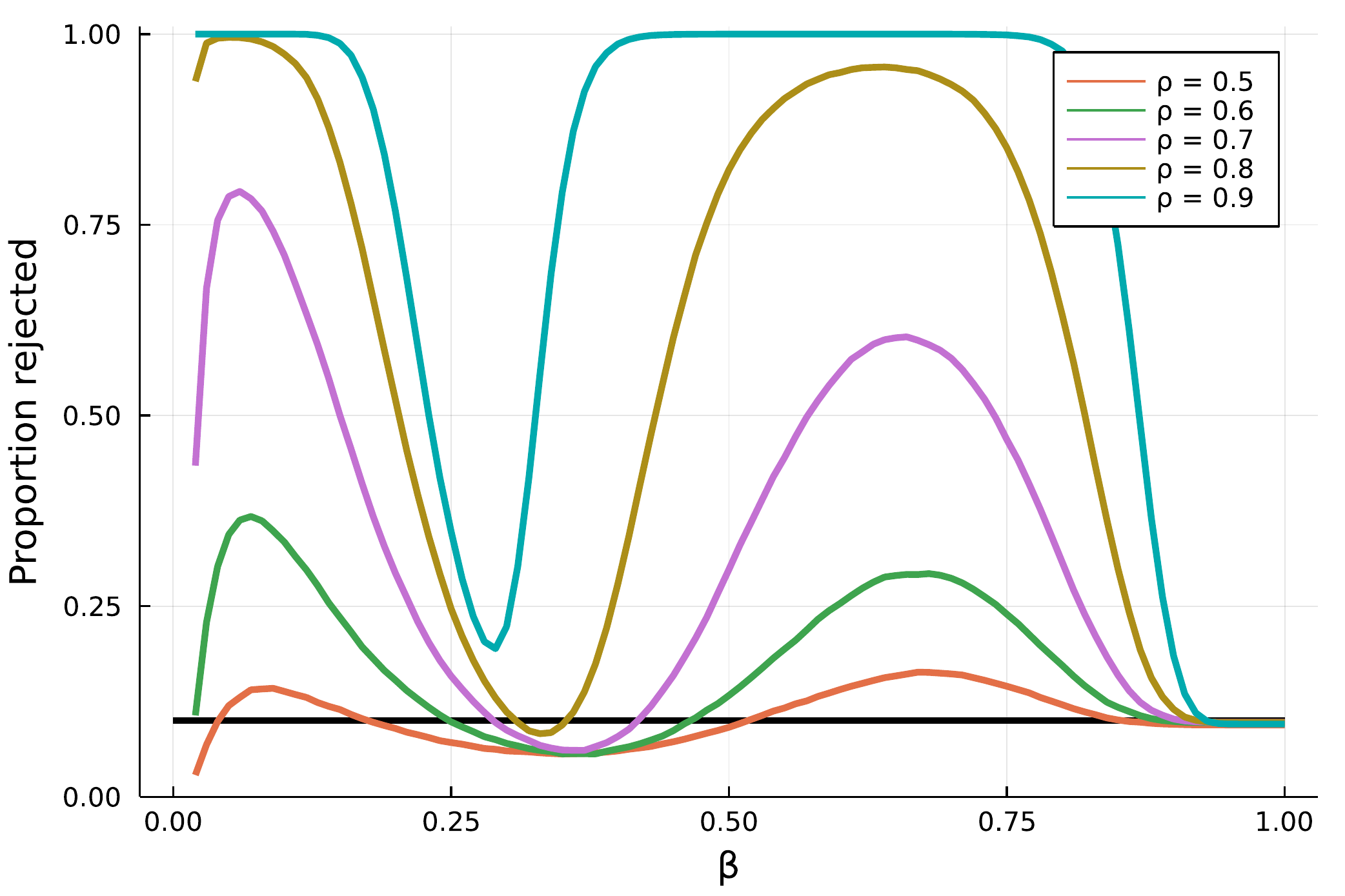}
\caption{Example 1: Detection of magnitude based sorting of a hidden column in tabular data. Results for $n=1,000$, $\alpha = 0.1$ (black line), and several levels of $\rho$. The proportion of cases detected ($H_0$ rejected) is plotted as a function of the sorting level $\beta$.}
\label{fig:proxy-sort-100}
\end{figure}

Assume now that we are presented with $x_1,\ldots,x_n$, the variables of a single column from the data and we wish to test if the data was sorted based on an unobserved column, say $\tilde{x}_1,\ldots,\tilde{x}_n$. Here correlations between columns may leave ``traces'' of the sorting of the unobserved $\tilde{x}_1,\ldots,\tilde{x}_n$ in the observed $x_1,\ldots,x_n$. The KS test will not detect such sorting since it is oblivious to the order of the data, but in many cases WW will do well. However, assume further that $\tilde{x}_1,\ldots,\tilde{x}_n$ is sorted based on the magnitude $|\tilde{x}_i|$. In this case, WW fails as well because it is based on runs around the median. Nevertheless, as we demonstrate, with a suitably chosen sorting level $\beta$, the partial bubble sorting test procedure works well.

To experiment with this scenario we consider Gaussian random vectors $(\tilde{X}, X)$ with zero mean and unit variance for both $\tilde{X}$ and $X$. The correlation coefficient between $\tilde{X}$ and $X$ is set at $\rho$. We generate $n$ repeated observations from this random vector and then sort the observations according to $|\tilde{X}|$. Now taking the second coordinate of the sorted data yields the resulting random vector $X_1,\ldots,X_n$ on which we execute the bubble sort testing procedure of Section~\ref{sec:gen-procedure}. In the extreme case of $\rho = 0$ the data $X_1,\ldots,X_n$ is iid and $H_0$ holds. However, as the magnitude of the correlation, $|\rho|$ grows towards $1$, traces of the sorting of $|\tilde{X}|$ may start to be detected in the data.

Figure~\ref{fig:proxy-sort-100} presents numerical results from an experiment with such data. In this case we considered samples of size $n=1,000$ and several correlation levels, $\rho$, ranging from $\rho=0.5$ which is hardly detectable to $\rho=0.9$ which is easily detectable. For each $\rho$, we ran $30,000$ Monte Carlo repetitions where for each repetition we repeated over a grid of sorting levels~$\beta$. We also used the same data for the WW test and for all $\rho$ values were not able to reject $H_0$ more than $10\%$ ($\alpha$) of the runs. Further, the KS test ($\beta = 1$) also fails. Interestingly, the bubble sort based procedure works well for some $\beta$ values and not for others with a dip in the ability to detect for $\beta \approx 0.3$. Our key point here though is that by choosing a suitable sorting level $\beta$, the sorting of the of the hidden column can be detected via the bubble sort test procedure but not via KS or WW.

\subsection*{Example 2: Service Time Dependent Scheduling in Queues}

We now consider a queueing system with a single server which serves jobs one by one when they are present, or idles otherwise. For general concepts of queueing models, and statistics of queues see for example \cite{asanjarani2021survey}, or references there-in. In our context we consider a finite set of job arrivals to the queue and observe the sequence of service durations of processed jobs, $X_1,\ldots,X_n$, where $X_1$ is the service duration of the first job served, $X_2$ is the service duration of the second job served, and so fourth. 

Apriori, before being served in the queue, the jobs arriving to the system have iid service durations with a known distribution $F_0$. The scheduling policy used to determine which job to serve next considers the available jobs in the queue at any time at which the server is free. This policy can affect the system behaviour and the nature of $X_1,\ldots,X_n$. With scheduling policies such as first come first serve, last come first serve, or random order of service, the output sequence of processed job durations $X_1,\ldots,X_n$ is iid because these policies do not depended on the actual service time of a job for scheduling. These policies and all other policies which induce iid service durations of processed jobs constitute $H_0$ in this example

Other policies may yield non-independent $X_1,\ldots,X_n$ and all such policies constitute $H_1$. One such policy is a policy where the server always selects the job with the smallest duration from the queue. In cases where there is occasional congestion in the system and multiple jobs wait in the queue simultaneously, this policy yields non-independent sequences of service times $X_1,\ldots,X_n$. To see this consider the extreme example in which all jobs arrive to the queue at the same time. In such a case the resulting sequence is completely sorted. Such a scheduling policy can generally be detected well by the WW test or the bubble sort testing procedure (with $\beta < 1)$, but not by the KS test.

\begin{figure}
\centering
\includegraphics[width=10cm]{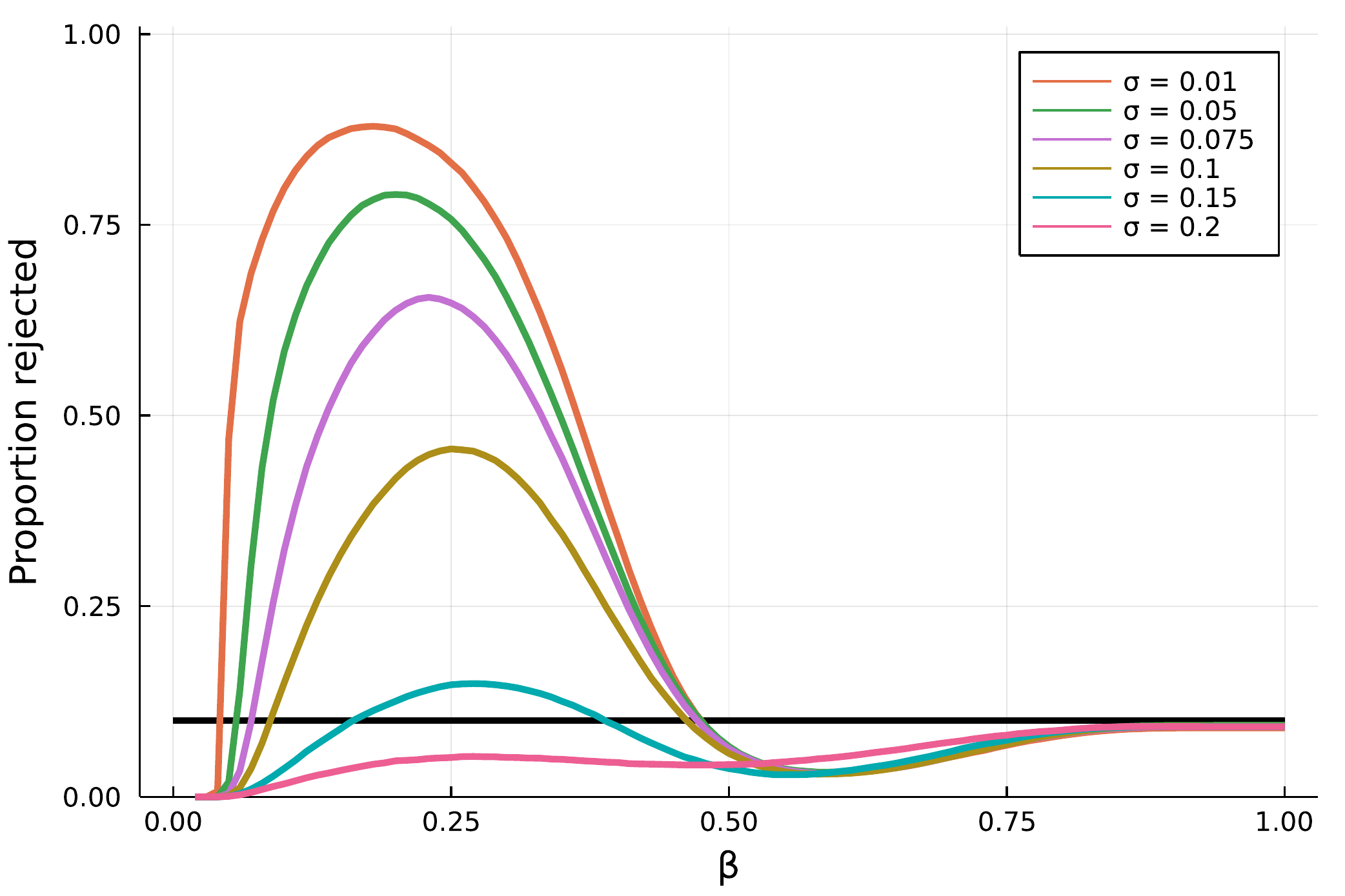}
\caption{Example 2: Service Time Dependent Scheduling in Queues. 
Results for $n=100$, $\alpha = 0.1$ (black line), and several levels of $\sigma$. The proportion of cases detected ($H_0$ rejected) is plotted as a function of the sorting level $\beta$.}
\label{fig:queue100}
\end{figure}

With other scheduling policies in $H_1$ it can be more difficult to confirm that the service requirements of the output process are not independent. We now focus on an experiment with one such policy. The policy is that the server randomly selects either the smallest or the largest job in the queue to go into service. This implies that the service requirements of the output process are dependent but it also yields sequences that are more difficult for the WW test to detect. Nevertheless, as we show via a numerical example, for suitably chosen sorting level $\beta$ the bubble sort based procedure exhibits solid power.

In our example we resort to an arrival process where the arrival times jobs are iid random variables which each follow the law of $e^{A}$, where $A$ has a normal distribution with expectation $\log(n)$ and a variance $\sigma^{2}$. As $\sigma \to 0$ we get that all jobs arrive simulataniously and as $\sigma$ grows we get a bigger spread between jobs implying lower congestion. We resort to $n=100$ jobs so the mean of $A$ is at around $4.6$. We take service durations of the jobs arriving to the queue as iid standard uniform random variables. With these parameters, to obtain some congestion in the queue we test levels of $\sigma$ at $0.01, 0.05, 0.075$, and $0.1$. To compare we also consider $\sigma = 0.15$ and $\sigma = 0.2$ at which point queue congestion is much less common. 

For each $\sigma$, we ran $100,000$ Monte Carlo repetitions where for each repetition we repeated over a grid of sorting levels~$\beta$. As in Example~1 we set $\alpha=0.1$ and also carry out the WW test on the simulated data. The WW rejects $H_0$ with estimated probabilities, $0.209, 0.203, 0.200, 0.184, 0.123$ and $0.100$ for the corresponding $\sigma$ values ranging from $0.01$ to $0.2$. In contrast, Figure~\ref{fig:queue100} presents the estimated power of the test as a function of the sorting level and for different values of $\sigma$. It is evident that in this example for $\beta \approx 0.25$ the bubble sort based procedure performs better than WW as long as there is some congestion in the queue ($\sigma$ is not too high).

\section{Discussion and Concluding Remarks}
\label{sec:discussion}

We demonstrated that partial sorting can be used to design test statistics that are sensitive to both the order and the underlying distribution of the data. More specifically, we showed that applying a fixed number of bubble sort iterations to a data set with iid elements gives rise to a limiting curve that we called the bubble sort curve. This curve can be regarded as a generalization of the cdf and shares similar properties. We then turned to the uniform distance between the empirical bubble sort curve and the limiting bubble sort curve. Analogous to the Kolmogorov--Smirnov test statistic, we used this distance to define the bubble sort test statistic. We characterized its asymptotic distribution as a generalized Kolmogorov--Smirnov distribution that only depends on the sorting level. We showed in examples that the bubble sort test statistics may outperform the classical test statistics if data are not obtained from independent observations.

There are several straightforward extensions of the current work. We focused in this paper on the one-sample bubble sort test statistic, but extending the theory to a two-sample test statistic is elementary. Another extension involves the left tail of the data. In the current setup, the largest data points are guaranteed to be fully sorted. Roughly speaking, this means that the right tail of the data is compared to the right tail of the underlying cdf in the bubble sort test statistic. Something similar can be achieved for the left tail by multiplying all data point by -1 or by changing the bubble sort algorithm to sort from right to left instead of from left to right.

There are also several new research directions suggested by our results. Two main questions stand out. The first question is what happens if we use a different sorting algorithm. This may give rise to a different limiting curve and a different test statistic with potentially better statistical properties. We suspect that less efficient sorting algorithms may be beneficial here. The second question is whether we can combine bubble sort test statistics with different sorting levels $\beta_{1} < \beta_{2} < \dotsb < \beta_{m}$ for the same data set. The examples showed that there is no single best sorting level $\beta$, but combining different sorting levels may lead to a more powerful test. 

\bibliographystyle{amsplain}
\bibliography{ref}

\vfill

\footnotesize
\noindent KB’s research was funded by SNSF Grant 200021-196888. HMJ and YN were partly funded by Australian Research Council (ARC) Discovery Project DP180101602.
	
\end{document}